 \documentclass[a4paper,11pt]{article}
 \usepackage[plainpages=false]{hyperref}
 \usepackage{amsfonts,amsmath,amssymb,amsthm}
 \usepackage{latexsym,lscape,rawfonts,mathrsfs}

 \usepackage[dvips]{color}
 \usepackage{multicol}

 \usepackage[numbers]{natbib}

 \usepackage[all]{xy}
 \usepackage{eufrak}
 \usepackage{makeidx}
 \usepackage{graphicx,psfrag}

 \usepackage{array,tabularx}

 \usepackage{setspace}

 \usepackage{appendix}

 \newcommand{\A}{\ensuremath{\mathbb{A}}}
 
 \newcommand{\D}[2]{\ensuremath{ \frac{\partial{#1}}{\partial{#2}}}}

 \newcommand{\Z}{\ensuremath{\mathbb{Z}}}
 \newcommand{\CP}{\ensuremath{\mathbb{CP}}}
 
 \newcommand{\st}{\ensuremath{\sqrt{-1}}}

 \newcommand{\ddb}{\ensuremath{\partial \bar{\partial}}}

 \newcommand{\KRf}{K\"ahler Ricci flow\;}
 \newcommand{\KRfc}{K\"ahler Ricci flow,\;}
 \newcommand{\KRfd}{K\"ahler Ricci flow.\;}
 \newcommand{\KRF}{K\"ahler Ricci Flow\;}

 \DeclareMathOperator{\Vol}{Vol}
 \DeclareMathOperator{\diam}{diam}

 \newcommand{\Blow}[1]{\ensuremath{\mathbb{CP}^2 \# {#1}\overline{\mathbb{CP}}^2}}

 \newcommand{\norm}[2]{{ \ensuremath{\|} #1 \ensuremath{\|}}_{#2}}
 \newcommand{\snorm}[2]{{ \ensuremath{\left |} #1 \ensuremath{\right |}}_{#2}}
 \newcommand{\sconv}{\ensuremath{ \stackrel{C^\infty}{\longrightarrow}}}

 \makeatletter
 \def\ExtendSymbol#1#2#3#4#5{\ext@arrow 0099{\arrowfill@#1#2#3}{#4}{#5}}
 
 \makeatother

 \definecolor{hao}{rgb}{1,0.5,0}
 \definecolor{miao}{cmyk}{0.5,0,0.2,0.2}
 \definecolor{qiao}{gray}{0.96}

 \newcommand{\Holder}{H\"{o}lder\;}
 
 \newcommand{\Poincare}{Poincar\`{e}\;}

 \newtheorem{claim}{Claim}
 \newtheorem*{clm}{Claim}
 
 \newtheorem{corollary}{Corollary}[section]
 \newtheorem{proposition}{Proposition}[section]
 \newtheorem{lemma}{Lemma}[section]
 
 \newtheorem{theorem}{Theorem}[section]
 
 \newtheorem*{thm}{Theorem}
 \newtheorem{definition}{Definition}[section]

 \newtheorem{remark}{Remark}[section]
 \newtheorem*{rmk}{Remark}

 \newtheorem{theoremin}{Theorem}
 \newtheorem{lemmain}{Lemma}
 
 \newtheorem{definitionin}{Definitionin}

\title{\KRf on Fano manfiolds(I)}
\author{Xiuxiong Chen\footnote{Partially supported by a NSF grant.}\;,  Bing Wang}
\date{}
\begin{document}

 \maketitle

\begin{abstract}
   We study the evolution of anticanonical line bundles along the
   \KRfd   We show that under some conditions, the convergence of
   \KRf is determined by the properties of the anticanonical divisors
   of $M$.     As examples, the \KRf on $M$ converges when $M$ is
   a Fano surface and $c_1^2(M)=1$ or $c_1^2(M)=3$.
   Combined with the work in~\cite{CW1}  and~\cite{CW2},
   this gives a Ricci flow proof of the Calabi conjecture on
   Fano surfaces with reductive automorphism groups.  The original proof
   of this conjecture is due to Gang Tian in \cite{Tian90}.
\end{abstract}

 \tableofcontents

\section{Introduction}

 In this paper, we introduce a new criteria for the convergence of the K\"ahler Ricci flow
 in   general Fano manifolds. This might be useful in attracting renewed attentions
 to the renown Calabi conjecture in higher dimensional Fano manifolds.  Moreover,  we  verify  these criteria for
  the K\"ahler Ricci flow in  Fano surfaces  $ \Blow{8}$ and $\Blow{6}$.
 Consequently, we  give a proof of the
 convergence of the \KRf  on such Fano surfaces.
 The existence of KE (K\"ahler Einstein) metrics on these Fano surfaces follows as a corollary. \\

 In K\"ahler geometry, a dominating problem is to prove the celebrated
 Calabi conjecture (\cite{Ca}).  It states that if the first Chern class of a K\"ahler manifold $M$
  is positive, null or negative, then the canonical K\"ahler class of $M$  admits
 a KE metric.   In 1976,  the null case Calabi conjecture was proved by S. T. Yau.
 Around the same time, the negative
 case was proved independently by T. Aubin and S. T. Yau.  However,
 the positive first Chern class  case is much more complicated.
 In~\cite{Ma}, Matsushima showed that the reductivity of $Aut(M)$
 is a necessary condition for the existence of KE metric.
 In \cite{Fu}, A. Futaki introduced an algebraic invariant which
 vanishes if the canonical K\"ahler class admits a KE  metric.
 Around 1988,  G. Tian~\cite{Tian90} proved the Calabi conjecture for Fano surfaces
  with reductive automorphism groups.    By the classification theory of complex surfaces,
  Tian actually proved the existence of KE metric on Fano surface $M$
   whenever $M$ is diffeomorphic to  $\Blow{k} \; (3 \leq k \leq 8)$.
 Prior to Tian's work, there is a series of
 important works in \cite{Tian87}, \cite{TY}, \cite{Siu} where
 existence results of KE metrics  on  some special complex surfaces
 were derived.\\

 Let $(M, [\omega])$ be a Fano manifold where $[\omega]$ is the
 canonical K\"ahler class. Suppose that $\{\omega_{t}\} (t \in [0,\infty))$ is the one parameter family
 of K\"ahler metrics in
 $[\omega]$ evolves under the K\"ahler Ricci flow.  Let $K_{M}^{-1}$
 be the anticanonical line bundle equipped with a natural, evolving metric $h_{t} =\omega_{t}^{n}=\det g_{\omega_t}.\;$
  In this paper, we adopt the view that one must study the the associated evolution of line bundles
 when study \KRf:
   \[
   \left(\begin{array} {c} (K_{M}^{-\nu}, h_{t}^{\nu}) \\
   \downarrow \\
   (M, \omega_{t}) \end{array}\right).
   \]
  It is well known that  $K_M^{-\nu}$ is very ample when $\nu$ is large.
  Let $N_{\nu}= \dim H^0(K_M^{-\nu})-1, \;$
  $\{S_{\nu, \beta}^t\}_{\beta=0}^{N_{\nu}}$ be orthonormal holomorphic sections in $H^0(K_M^{-\nu})$
  with respect to metric $g_t$ and $h_t^{\nu}$, i.e.,
  \begin{align*}
      \int_M \left< S_{\nu, \alpha}^{t}, S_{\nu, \beta}^t \right>_{h_t^{\nu}} \omega_{t}^{n} =
      \delta_{\alpha \beta}.
  \end{align*}
    An easy observation shows that
  \begin{align*}
   F_{\nu}(x, t) =   \frac{1}{\nu} \log \sum_{\beta=0}^{N_{\nu}} \snorm{S_{\nu, \beta}^t}{h_t^{\nu}}^2(x)
  \end{align*}
  is a well defined function on $M \times [0, \infty)$ (independent of the choice of orthonormal basis).
  The  \KRf $\{(M, g(t)), 0 \leq t < \infty\}$ is called a {\bf flow tamed by $\nu$} if
  $K_M^{-\nu}$ is very ample and $F_{\nu}(x, t)$ is a uniformly bounded function on $M \times [0,
  \infty)$.  The flow is called a {\bf tamed \KRf} if it is tamed by a some integer $\nu$.

   For a tamed flow, we can reduce the convergence of the flow to the values of local
   $\alpha$-invariants of plurianticanonical divisors.

 \begin{definitionin}
   Suppose $L$ is a line bundle over $M$ with Hermitian metric $h$,  $S$ is a holomorphic
   section of $L$, $x \in M$.  Define
 \begin{align*}
        \alpha_x(S)= \sup\{\alpha|   \norm{S}{h}^{-2\alpha} \; \textrm{is locally integrable around} \;
        x\}.
 \end{align*}
 \label{definitionin: lalpha}
 \end{definitionin}
 See~\cite{Tian90}  and~\cite{Tian91} for more details about this
 definition.
 Note that $\alpha_x(S)$ is also called singularity exponent (\cite{DK}),
 logarithm canonical threshold (\cite{ChS}), etc.
 It is determined only by the singularity type of $Z(S)$.
 Therefore, if $S \in H^0(K_M^{-\nu})$, $\alpha_x(S)$ can only
 achieve finite possible values.

 \begin{definitionin}
  Let $\mathscr{P}_{G, \nu, k}(M, \omega)$ be the collection of all
  $G$-invariant functions of form
  $\displaystyle  \frac{1}{\nu} \log (\sum_{\beta=0}^{k-1} \norm{\tilde{S}_{\nu, \beta}}{h^{\nu}}^2)$,
  where $\{\tilde{S}_{\nu, \beta}\}_{\beta=0}^{k-1} \; ( 1 \leq  k \leq \dim H^0(K_M^{-\nu}))$
  satisfies
  \begin{align*}
     \int_M \langle \tilde{S}_{\nu, \alpha}, \tilde{S}_{\nu, \beta}
     \rangle_{h^{\nu}} \omega^n = \delta_{\alpha \beta},
     \quad
     0 \leq \alpha, \beta \leq k-1 \leq \dim (K_M^{-\nu}) -1;
     \quad
     h= \det g_{\omega}.
  \end{align*}

    Define
  \begin{align*}
      \alpha_{G, \nu, k} \triangleq
      \sup\{ \alpha |  \sup_{\varphi \in \mathscr{P}_{G,\mu, k}} \int_M e^{-\alpha \varphi} \omega^n <
      \infty\}.
  \end{align*}
  If $G$ is trivial, we denote $\alpha_{\nu, k}$ as $\alpha_{G, \nu, k}$.
  \label{definitionin: nualpha}
 \end{definitionin}

   It turns out that  the value
  of local $\alpha$-invariants,
  $\alpha_{\nu, 1}$ and $\alpha_{\nu, 2}$  play important roles in the convergence of K\"ahler Ricci flow.

  \begin{theoremin}
  Suppose $\{(M^n, g(t)), 0 \leq t < \infty \}$ is a \KRf tamed by $\nu$.
  If $\alpha_{\nu, 1}> \frac{n}{(n+1)}$,
  then $\varphi$ is uniformly bounded along this flow. In
  particular, this flow converges to a KE metric exponentially fast.
 \label{theoremin: nuconv}
 \end{theoremin}

 \begin{theoremin}
 Suppose $\{(M^n, g(t)), 0 \leq t < \infty \}$ is a \KRf tamed by $\nu$.
 If $\alpha_{\nu, 2}>\frac{n}{n+1}$ and
 $\alpha_{\nu, 1} > \frac{1}{2- \frac{n-1}{(n+1) \alpha_{\nu,2}}}$, then $\varphi$ is
 uniformly bounded along this flow.  In particular, this flow converges to a KE
 metric exponentially fast.
 \label{theoremin: nuconvr}
 \end{theoremin}

  In fact, if a \KRf is tamed by some large $\nu$, an easy argument (c.f. Section 2.3) shows that
  the following strong partial $C^0$-estimate hold.
   \begin{align}
     \left|\varphi(t) - \sup_M \varphi(t)
    -\frac{1}{\nu}\log \sum_{\beta=0}^{N_{\nu}} \snorm{\lambda_{\beta}(t) \tilde{S}_{\nu, \beta}^t}{h_0^{\nu}}^2 \right| < C.
   \label{eqn: spe}
   \end{align}
  Here $\varphi(t)$ is the evolving K\"ahler potential.
  $0 < \lambda_0(t) \leq \lambda_1(t) \leq \cdots \leq \lambda_{N_{\nu}}(t)=1 $
  are $N_{\nu}+1$ positive functions of time $t$.
  $\{\tilde{S}_{\nu, \beta}^t\}_{\beta=0}^{N_{\nu}}$ is an orthonormal
  basis of $H^0(K_M^{-\nu})$ under the fixed metric $g_0$.
  Intuitively, inequality (\ref{eqn: spe}) means that we can control $Osc_{M} \varphi(t)$ by
  $\displaystyle \frac{1}{\nu}\log \sum_{\beta=0}^{N_{\nu}} \snorm{\lambda_{\beta}(t) \tilde{S}_{\nu, \beta}^t}{h_0^{\nu}}^2$
  which only blows up along intersections of pluri-anticanonical divisors.
   Therefore, the estimate of $\varphi(t)$ is more or less translated to the study of the property of
  pluri-anticanonical holomorphic sections.\\

 In view of these theorems, we need to check the following two conditions:
 \begin{itemize}
 \item Whether the K\"ahler Ricci flow is a tamed flow;
 \item Whether the $\alpha_{\nu, k} \; (k=1,2)$ are big enough.
 \end{itemize}
 The second condition can be checked by purely algebraic geometry
 method.   The first condition is much weaker.
 We believe that it holds for every \KRf on
 Fano manifold although we cannot prove this right now.
 However, under some extra conditions, we can check the
 first condition by the following theorem.

  \begin{theoremin}
   Suppose $\{(M^n, g(t)), 0 \leq t < \infty\}$ is a \KRf satisfying
   the following conditions.
   \begin{itemize}
   \item volume ratio bounded from above, i.e., there exists a
   constant $K$ such that
   \begin{align*}
   \Vol_{g(t)}(B_{g(t)}(x, r)) \leq Kr^{2n}
   \end{align*}
   for every geodesic ball $B_{g(t)}(x, r)$ satisfying $r \leq 1$.
   \item weak compactness, i.e., for every sequence $t_i \to \infty$, by
   passing to subsequence, we have
   \begin{align*}
      (M, g(t_i)) \sconv (\hat{M}, \hat{g}),
   \end{align*}
   where $(\hat{M}, \hat{g})$ is a Q-Fano normal variety,
   $\sconv$ means Cheeger-Gromov convergence, i.e., $(M, g(t_i))$ converges to $(\hat{M}, \hat{g})$
   in Gromov-Hausdorff topology, and the convergence is in smooth topology away
   from singularities.
   \end{itemize}
 Then this flow is tamed.
  \label{theoremin: justtamed}
  \end{theoremin}

 In the case of Fano surfaces under continuous path, Tian proved a similar theorem.
 However, in his proof, every metric is a K\"ahler Einstein metric, so there are more
 estimates available. In particular, Ricci curvature is uniformly bounded there.
 Our proof here is more technical since we have no Ricci curvature control.
  The concept of Q-Fano variety is first defined in~\cite{DT}, it's
  a natural generalization of Fano manifold. A Q-Fano variety is an algebraic variety
  with a very ample line bundle whose restriction on the smooth part is the plurianticanonical line bundle.
  In the proof of  this theorem, H\"ormander's $L^2$-estimate of
  $\bar{\partial}$-operator, Perelman's fundamental estimates and
  the uniform control of Sobolev constants (c.f.~\cite{Ye},~\cite{Zhq}) play crucial roles.
  Actually,  Sobolev constants control and Perelman's estimates
  assure the uniform control of $\norm{|S|_{h_t^{\nu}}}{C^0(M)}$ and
  $\norm{|\nabla S|_{h_t^{\nu}}}{C^0(M)}$ whenever $S$ is a  unit norm
  holomorphic section in $H^0(K_M^{-\nu})$.  H\"ormander's
  $L^2$-estimate of $\bar{\partial}$-operator assures that the
  plurigenera is continuous under the sequential convergence.
  Therefore, for every fixed $\nu$, we have
  \begin{align*}
        \lim_{i \to \infty} \inf_{x \in M} e^{\nu F_{\nu}(x, t_i)}
  =\lim_{i \to \infty}  \inf_{x \in M} \sum_{\beta=0}^{N_{\nu}}
  \snorm{S_{\nu, \beta}^{t_i}}{h_{t_i}^{\nu}}^2(x)
  =\inf_{x \in \hat{M}}
     \sum_{\beta=0}^{N_{\nu}} \snorm{\hat{S}_{\nu,
     \beta}}{\hat{h}^{\nu}}^2(x).
  \end{align*}
  This equation relates the tamed condition to the property of every
  limit space.  If every limit space is a Q-Fano normal variety,
  we know
  \begin{align*}
   \inf_{x \in \hat{M}}
     \sum_{\beta=0}^{N_{\nu}} \snorm{\hat{S}_{\nu,
     \beta}}{h_{t_i}^{\nu}}^2(x)>0
  \end{align*}
  for some $\nu$ depending on $\hat{M}$.  Then a contradiction
  argument can show that $e^{\nu F_{\nu}}$ must be uniformly bounded
  from below for some large $\nu$. In other words, $F_{\nu}$ is
  uniformly bounded (the upper bound of $F_{\nu}$ is a corollary of the boundedness of
  $\norm{|S|_{h_t^{\nu}}}{C^0(M)}$) and the flow is tamed.\\

 As applications of Theorem~\ref{theoremin: nuconv} to
 Theorem~\ref{theoremin: justtamed}, we can show the convergence of
 \KRf on Fano surface $M$ when $c_1^2(M) \leq 4$.   Actually,  in~\cite{CW3},
 we proved the weak compactness of 2-dimensional \KRfd

 \begin{lemmain}(\cite{CW3})
    Suppose $\{ (M,g(t)), 0 \leq t < \infty \}$ is a K\"ahler Ricci
   flow solution on a Fano surface.  Then for any sequence
   $t_i \to \infty$, we have Cheeger-Gromov convergence
 \begin{align*}
   (M, g(t_i)) \stackrel{C^\infty}{\to} (\hat{M}, \hat{g})
 \end{align*}
 where $(\hat{M}, \hat{g})$ is a K\"ahler Ricci soliton orbifold
 with finite singularities. In particular, $\hat{M}$ is Q-Fano normal
 variety.
 \label{lemmain: weakcompactness}
 \end{lemmain}

  Moreover, the volume ratio upper bound is proved in the process of proving weak compactness.
 Therefore, Theorem~\ref{theoremin: justtamed} and Lemma~\ref{lemmain: weakcompactness}
 implies that
  every 2-dimensional \KRf is a tamed flow.  The authors remark that,
 in an unpublished work (c.f.~\cite{Se1},~\cite{FZ}), Tian has pointed out earlier
 the sequential convergence of the 2-dimensional \KRf to  K\"ahler
 Ricci soliton orbifolds under the Gromov-Hausdorff topology.   Under the extra condition that Ricci curvature
 is uniformly bounded along the flow, Lemma~\ref{lemmain: weakcompactness} was proved by Natasa Sesum in~\cite{Se1}.
 However, for our purpose of using Theorem~\ref{theoremin: justtamed},
 these convergence theorems are not sufficient (We need Cheeger-Gromov
 convergence without Ricci curvature bound condition).
 In the course of proof of this lemma, the fundamental work of
 G. Perelman on the Ricci flow (non-local collapsing theorem, pseudo-locality theorem and canonical
 neighborhood theorem)play critical roles.  Under some geometric constraints natural to our setting,
 we proved an inverse pseudo-locality theorem (heuristically speaking,
 no ``bubble" will disappear suddenly).   For that purpose, we need to have a uniform control of volume growth on
 all scales.  We found that the argument for volume ratio upper bound in the beautiful
  work [TV1] and [TV2] is very enlightening.\\

 In order to show the convergence of a 2-dimensional \KRfc
 we now only need to see if $\alpha_{\nu, k} (k=1, 2)$ are big enough to satisfy the
 requirements of Theorem~\ref{theoremin: nuconv} or Theorem~\ref{theoremin: nuconvr}.
 If $c_1^2(M) \leq 4$, one can show that either
 Theorem~\ref{theoremin: nuconv} or Theorem~\ref{theoremin: nuconvr}
 applies.   However, the convergence of \KRf on Fano surfaces $M$ are proved
 in~\cite{CW2} when $c_1^2(M)=2$ or $c_1^2(M)=4$.
 The only remained cases are $c_1^2(M)=1$ and $c_1^2(M)=3$.  So we
 concentrate on these two cases and have the following lemma.

 \begin{lemmain}
  Suppose $M$ is a Fano surface, $\nu$ is any positive integer.
  \begin{itemize}
  \item  If $c_1^2(M)=1$, then  $\alpha_{\nu, 1} \geq \frac{5}{6}$.
  \item  If $c_1^2(M)=3$, then  $\alpha_{\nu, 1} \geq \frac{2}{3}$,
      $\alpha_{\nu, 2}> \frac{2}{3}$.
  \end{itemize}
  \label{lemmain: localalpha}
 \end{lemmain}

 Actually, the value of $\alpha_{\nu, 1}$ was calculated by Ivan Cheltsov (c.f.~\cite{Chl})
 for every Fano surface.  The value of $\alpha_{\nu, 2}$  was also calculated
 for every cubic surface ($c_1^2(M)=3$) by Yalong Shi (c.f.~\cite{SYl}).
 For the convenience of readers, we give an elementary proof at the end of this paper.\\

 Therefore, Theorem~\ref{theoremin: nuconv} and
 Theorem~\ref{theoremin: nuconvr} applies respectively and show the
 existence of KE metrics on $M$ whenever $c_1^2(M)=1$ or $3$.
 Combining this result with the results we proved in~\cite{CW1}
 and~\cite{CW2}, we can give an alternative proof of the celebrated theorem of Tian:

  \begin{thm}[\cite{Tian90}]
  A Fano surface $M$ admits a K\"ahler Einstein metric  if and only if
   $Aut(M)$ is reductive.
  \label{thm: tian90}
  \end{thm}

This solved a  famous problem of Calabi for Fano surfaces \cite{Ca}.
 This work of Tian clearly involves deep understanding of many aspects of K\"ahler
 geometry as well as its intimate connection to algebraic geometry.
 It is one of the few highlights in
 K\"ahler geometry which deserve new proofs by Ricci flow.   On the other hand, the K\"ahler Ricci flow
 is  a natural way to understand Calabi conjecture in Fano setting.
 Following Yau's estimate (\cite{Yau78}),
   H. D. Cao (\cite{Cao85})
  proved that  the \KRF with smooth initial metric always exists globally.
  On a KE manifold, the first named author and Tian
  showed that \KRF converges exponentially fast toward the KE
  metric if the initial metric has positive bisectional
  curvature (c.f.~\cite{CT1},~\cite{CT2}).
  Using his famous $\mu$-functional,  Perelman proved that
  scalar curvature, diameter and normalized Ricci potentials are all uniformly bounded
  along \KRf (c.f.~\cite{SeT}).  These fundamental estimates of G. Perelman opens the door for
  a more qualitative analysis of singularities formed in the \KRfd
  As a corollary of his estimates, G. Perelman announced that the \KRf
  will always converge to the KE metric on every KE manifold.
  The first written proof of this statement appeared in \cite{TZ} where
  Tian and Zhu also generalized it to K\"ahler manifolds admitting K\"ahler Ricci solitons.
  In our humble view,
  the  estimates  of  G. Perelman  makes the flow approach a plausible one in terms of understanding
  Calabi conjecture in Fano setting.  We hope that this modest progress in K\"ahler Ricci flow
  will attract more attentions to the renown Hamilton-Tian conjecture. Namely, any K\"ahler Ricci flow
  will converge to some K\"ahler Ricci solitons with mild singularities in Cheeger-Gromov topology,
   perhaps of different complex structures. \\

   The application of strong partial $C^0$-estimate is one of the crucial components of this
   paper.  It sets up the frame
   work of our proof for the convergence of 2-dimensional \KRfd
   This estimate originates from the strong partial $C^0$-estimate along continuous path in Tian's original
   proof (c.f.~\cite{Tian90}).
   He also conjectured that the strong partial $C^0$-estimate holds along continuous path in higher
    dimensional K\"ahler manifolds (c.f.~\cite{Tian91}).\\

 The organization of this paper is as follows. In section 2,
 along each tamed flow, we reduce the
 $C^0$-estimate of the potential function $\varphi$ to the calculation of
 local $\alpha$-invariants of sections $S \in H^0(K_M^{-\nu})$.
 In section 3, we study the basic properties of pluri-anticanonical holomorphic
 sections along \KRfd Here we discuss the applications of H\"{o}rmander's
  $L^2$-estimate of $\bar{\partial}$-operator and we deduce the uniform bounds of
  $\snorm{S}{h_t^{\nu}}$ and $\snorm{\nabla S}{h_t^{\nu}}$. Using these
  estimates, we give a justification theorem of the tamed condition.
 In section 4,  we calculate $\alpha_{\nu, k}(M) (k=1,2)$ when $c_1^2(M)=1$ or
 $3$ and show that their values are big enough to obtain the $C^0$-estimate
 of the evolving potential function $\varphi(t)$.\\

 \begin{rmk}
  In the subsequent paper~\cite{Wang}, we will apply these methods to \KRf on
  orbifold Fano surfaces. As an application, we find some new
  K\"ahler Einstein orbifolds.  In particular, we prove the following
  conjecture (c.f.~\cite{Kosta}).  Let $Y$ be a degree $1$ del Pezzo
  surface having only Du Val singularities of type $\A_n$ for $n \leq
  6$, then $Y$ admits a K\"ahler Einstein metric.
 \end{rmk}

 \noindent {\bf Acknowledgment}
 This work benefits from \cite{Tian90} conceptually. The second named author is very grateful
 to G. Tian for many insightful and inspiring conversations with him.
 He also  would like to thank J. Cheeger,
 B.  Chow, K. Grove, J.P. Bourguignon for their interests in this work.
 The first named author would like to thank
 S. K. Donaldson for  lengthy discussions in this and related projects in K\"ahler geometry.

 \section{Estimates Along \KRF}

 \subsection{Basic K\"ahler Geometry}
 Let $M$ be an $n$-dimensional compact K\"ahler manifold. A K\"ahler
 metric can be given by its K\"ahler form $\omega$ on $M$. In local
 coordinates $z_1, \cdots, z_n$, this $\omega$ is of the form
 \[
 \omega = \sqrt{-1} \displaystyle \sum_{i,j=1}^n\;g_{i \overline{j}}
 d\,z^i\wedge d\,z^{\overline{j}}  > 0,
 \]
 where $\{g_{i\overline {j}}\}$ is a positive definite Hermitian
 matrix function. The K\"ahler condition requires that $\omega$ is a
 closed positive (1,1)-form.  Given a K\"ahler metric $\omega$, its
 volume form  is
 \[
   \omega^n = {1\over {n!}}\;\left(\sqrt{-1} \right)^n \det\left(g_{i \overline{j}}\right)
  d\,z^1 \wedge d\,z^{\overline{1}}\wedge \cdots \wedge d\,z^n \wedge d\,z^{\overline{n}}.
 \]
 The curvature tensor is
 \[
  R_{i \overline{j} k \overline{l}} = - {{\partial^2 g_{i \overline{j}}} \over
 {\partial z^{k} \partial z^{\overline{l}}}} + \displaystyle
 \sum_{p,q=1}^n g^{p\overline{q}} {{\partial g_{i \overline{q}}}
 \over {\partial z^{k}}}  {{\partial g_{p \overline{j}}} \over
 {\partial z^{\overline{l}}}}, \qquad\forall\;i,j,k,l=1,2,\cdots n.
 \]
 The Ricci curvature form is
 \[
   {\rm Ric}(\omega) = \sqrt{-1} \displaystyle \sum_{i,j=1}^n \;R_{i \overline{j}}(\omega)
 d\,z^i\wedge d\,z^{\overline{j}} = -\sqrt{-1} \partial
 \overline{\partial} \log \;\det (g_{k \overline{l}}).
 \]
 It is a real, closed (1,1)-form and $[Ric]=2\pi c_1(M)$.\\

  From now on we we assume $M$ has positive first Chern class, i.e., $c_1(M)>0$.
  We call $[\omega]$ as a canonical K\"ahler class if $[\omega]=[Ric]=2\pi c_1(M)$.
    If we require the initial metric is in canonical class,
    then the normalized Ricci flow (c.f. \cite{Cao85}) on $M$ is
\begin{equation}
  {{\partial g_{i \overline{j}}} \over {\partial t }} = g_{i \overline{j}}
  - R_{i \overline{j}}, \qquad\forall\; i,\; j= 1,2,\cdots,n.
\label{eq:kahlerricciflow}
\end{equation}
  Denote $\omega = \omega_{g(0)}$, $\omega_{g(t)} = \omega + \st \ddb \varphi_t$.
 $\varphi_t$ is called the K\"ahler potential and sometime it is denoted as
 $\varphi$ for simplicity.  On the level of K\"ahler potentials, \KRf becomes
\begin{equation}
   {{\partial \varphi} \over {\partial t }}
   =  \log {{\omega_{\varphi}}^n \over {\omega}^n } + \varphi  + u_{\omega},
\label{eq:flowpotential}
\end{equation}
where $u_{\omega}$ is defined by
\begin{align*}
  {\rm Ric}(\omega)- \omega = -\sqrt{-1} \partial \overline{\partial} u_{\omega},
    \; {\rm and}\;\displaystyle \int_M \;
  (e^{-u_{\omega}} - 1)  {\omega}^n = 0.
 \end{align*}
 As usual, the flow equation (\ref{eq:kahlerricciflow}) or
 (\ref{eq:flowpotential}) is referred as the \KRf in canonical class
 of $M$. It is proved by Cao~\cite{Cao85}, who followed Yau's
 celebrated work~\cite{Yau78}, that this flow exists globally for
 any smooth initial K\"ahler metric in the canonical class.

In this note, we only study \KRf in the canonical class. For the
   simplicity of notation, we may not mention that the flow is in
   canonical class every time.

 Along \KRf, the evolution equations of curvatures are listed in
  Table~\ref{table: cevformula}.

 \begin{table}[h]
 \begin{center}
 \begin{tabular}[b]{|c @{} l|}
 \hline
 $\displaystyle \D{}{t}R_{i \overline{j} k
 \overline{l}}$ & =$\displaystyle \triangle R_{i \overline{j} k \overline{l}} +
 R_{i \overline{j} p \overline{q}} R_{q \overline{p} k \overline{l}}
 - R_{i \overline{p} k \overline{q}} R_{p \overline{j} q
 \overline{l}} + R_{i
 \overline{l} p \overline{q}} R_{q \overline{p} k \overline{j}} + R_{i \overline{j} k
 \overline{l}}$\\
  &\quad $-{1\over 2} \left( R_{i \overline{p}}R_{p \overline{j} k
 \overline{l}}  + R_{p \overline{j}}R_{i \overline{p} k \overline{l}}
 + R_{k \overline{p}}R_{i \overline{j} p \overline{l}} + R_{p
 \overline{l}}R_{i \overline{j} k \overline{p}} \right)$.\\
 $\displaystyle \D{}{t} R_{i\bar{j}}$  &  =$\displaystyle \triangle
  R_{i\bar j} + R_{i\bar j p \bar q} R_{q \bar p} -R_{i\bar p} R_{p \bar j}$.\\
 $\displaystyle \D{}{t} R$ & =$\displaystyle \triangle R + R_{i\bar j} R_{j\bar i}- R$.\\
 \hline
 \end{tabular}
 \end{center}
 \caption{Curvature evolution equations along \KRf}
 \label{table: cevformula}
 \end{table}

 Let $\mathscr{P}(M, \omega)= \{\varphi | \omega + \st \ddb \varphi \}$. It is shown in~\cite{Tian87} that there is a small constant $\delta>0$ such that
 \begin{align*}
     \sup_{\varphi \in \mathscr{P}(M, \omega)}
    \frac{1}{V} \int_M e^{-\delta(\varphi -\sup_M \varphi)} \omega^n < \infty.
 \end{align*}
 The supreme of such $\delta$ is called the $\alpha$-invariant of $(M, \omega)$ and
 it is denoted as $\alpha(M, \omega)$.
 Let $G$ be a compact subgroup of $Aut(M)$ and $\omega$ is a $G$-invariant form. We denote
 \begin{align*}
   \mathscr{P}_G(M, \omega)
   =\{\varphi | \omega + \st \ddb \varphi >0, \varphi \; \textrm{is invariant under} \; G
   \}.
 \end{align*}
 Similarly, we can define $\alpha_G(M, \omega)$.  Actually, $\alpha_G(M, \omega)$ is an algebraic invariant.
 It is called global log canonical threshold $lct(X, G)$ by algebraic geometers.
 See~\cite{ChS} for more details.

\subsection{Known Estimates along General \KRF}

 There are a lot of estimates along \KRf in the literature. We list some of them
 which are important to our arguments.

 \begin{proposition}[Perelman, c.f.~\cite{SeT}]
 Suppose $\{(M^n, g(t)), 0 \leq t < \infty \}$ is a \KRf solution. There
 are two positive constants $\mathcal{B}, \kappa$ depending only on this flow such that the
 following two estimates hold.
 \begin{enumerate}
 \item Under metric $g(t)$,  let $R$ be the scalar curvature,
 $-u$ be the normalized Ricci potential, i.e.,
 \begin{align*}
 Ric-\omega_{\varphi(t)}= - \st \ddb{u}, \quad
 \frac{1}{V} \int_M e^{-u} \omega_{\varphi(t)}^n=1.
 \end{align*}
 Then we have
 \begin{align*}
      \norm{R}{C^0} + \diam M +
      \norm{u}{C^0} + \norm{\nabla u}{C^0} < \mathcal{B}.
 \end{align*}
 \item  Under metric $g(t)$,  $ \displaystyle
      \frac{\Vol(B(x, r))}{r^{2n}} > \kappa$ for every
       $r \in (0, 1)$, $(x, t) \in M \times [0, \infty)$.
 \end{enumerate}
 \label{proposition: perelman}
 \end{proposition}

 After this fundmental work of G. Perelman, many interesting papers appear during this
   period. We include a few references here for the convenience of readers:  ~\cite{CH},~\cite{CST},
   ~\cite{CW1},~\cite{CW2},~\cite{Hei},~\cite{PSS},
  ~\cite{PSSW1},~\cite{PSSW2},~\cite{Ru},~\cite{RZZ},~\cite{Se1},~\cite{Se2},~\cite{TZs},etc.
  In this subsection, we cite a few results below which are directly related to our work here.

 \begin{proposition}[\cite{Zhq}, \cite{Ye}]
 There is a uniform Sobolev constant $C_S$ along the \KRf solution
 $\{(M^n, g(t)), 0 \leq t <\infty \}$. In other words, for every
 $f \in C^{\infty}(M)$, we have
 \begin{align*}
       (\int_M |f|^{\frac{2n}{n-1}}
       \omega_{\varphi}^n)^{\frac{n-1}{n}}
       \leq
       C_S\{\int_M |\nabla f|^2 \omega_{\varphi}^n
       + \frac{1}{V^{\frac{1}{n}}} \int_M |f|^2  \omega_{\varphi}^n\}.
 \end{align*}
 \label{proposition: sobolev}
 \end{proposition}

  \begin{proposition}[c.f. \cite{TZ}]
   There is a uniform weak Poincar\`e constant $C_P$ along the \KRf
   solution $\{(M^n, g(t)), 0 \leq t < \infty \}$. Namely, for every
   nonnegative function $f \in C^{\infty}(M)$, we have
   \begin{align*}
        \frac{1}{V} \int_M f^2 \omega_{\varphi}^n   \leq
          C_P\{\frac{1}{V} \int_M |\nabla f|^2 \omega_{\varphi}^n
           + (\frac{1}{V} \int_M f  \omega_{\varphi}^n)^2\}.
   \end{align*}
 \label{proposition: poincare}
 \end{proposition}

  As an easy application of a normalization technique initiated
  in~\cite{CT1}, one can prove the following property.

   \begin{proposition}[c.f. \cite{PSS}, \cite{CW2}]
    By properly choosing initial condition, we have
 \begin{align*}
 \norm{\dot{\varphi}}{C^0} + \norm{\nabla \dot{\varphi}}{C^0}<C
 \end{align*}
 for some constant $C$ independent of time $t$.
 \label{proposition: dotphi}
 \end{proposition}

  Based on these estimates, the authors proved the following properties.

 \begin{proposition}[c.f. \cite{Ru}, \cite{CW2}]
   There is a constant $C$ such that
 \begin{align}
   \frac{1}{V} \int_M (-\varphi) \omega_{\varphi}^n \leq
  n \sup_M \varphi  - \sum_{i=0}^{n-1}
     \frac{i}{V} \int_M \st \partial \varphi \wedge
     \bar{\partial} \varphi \wedge \omega^i \wedge
     \omega_{\varphi}^{n-1-i} + C.
  \label{eqn: dsupphi}
 \end{align}
 In particular, we have
   \begin{align}
       \frac{1}{V} \int_M (-\varphi) \omega_{\varphi}^n  \leq n \sup_M
       \varphi + C.
   \label{eqn: supphi}
   \end{align}
 \label{proposition: supphi}
 \end{proposition}

 \begin{proposition}[c.f. \cite{Ru}, \cite{CW2}]
  For every $\delta$ less than the $\alpha$-invariant of $M$,  there
  is a uniform constant $C$ such that
   \begin{align}
     \sup_M \varphi < \frac{1-\delta}{\delta}\int_M (-\varphi) \omega_{\varphi}^n  + C
      \label{eqn: deltainv}
   \end{align}
  along the flow.
  \label{proposition: nesup}
  \end{proposition}

 \begin{lemma}[c.f. \cite{Ru}, \cite{CW2}]
  Along \KRf $\{(M^n, g(t)), 0 \leq t < \infty \}$ in the canonical class of
  Fano manifold $M$, the following conditions are equivalent.
 \begin{itemize}
 \item $\varphi$ is uniformly bounded.
 \item $\displaystyle \sup_M \varphi$ is uniformly bounded from above.
 \item $\displaystyle \inf_M \varphi$  is uniformly bounded from below.
 \item $\int_M \varphi \omega^n$ is uniformly bounded from above.
 \item $\int_M (-\varphi) \omega_{\varphi}^n$ is
   uniformly bounded from above.
 \item $I_{\omega}(\varphi)$ is uniformly bounded.
 \item $Osc_{M} \varphi$ is uniformly bounded.
 \end{itemize}
 \label{lemma: conditions}
 \end{lemma}

 As a simple corollary, we have
  \begin{theorem}[\cite{CW2}]
    If $\alpha_G(M, \omega)> \frac{n}{n+1}$ for some $G$-invariant metric $\omega$,
  then $\varphi$ is uniformly  bounded along
    the \KRf initiating from $\omega$.
    \label{theorem: alphag}
 \end{theorem}

 \subsection{Estimates along Tamed \KRF}

 In this section, we only study tamed flow.

 \begin{definition}
 For every positive integer $\nu$, we can define a function $F_{\nu}$ on
 spacetime $M \times [0, \infty)$  as follows.
 \begin{align*}
    F_{\nu}(x, t) \triangleq \frac{1}{\nu} \log \sum_{\beta=0}^{N_{\nu}} \snorm{S_{\nu,
    \beta}^t}{h_t^{\nu}}^2(x)
 \end{align*}
 where $\{S_{\nu, \beta}^t\}_{\beta=0}^{N_{\nu}}$ is an orthonormal basis
 of $H^0(K_M^{-\nu})$ under the metric $g_t=g(t)$ and $h_t^{\nu}=(\det g_t)^{\nu}$, i.e.,
  \begin{align*}
  \displaystyle \int_M \langle S_{\nu, \alpha}^t, S_{\nu, \beta}^t \rangle_{h_t^{\nu}} \omega_t^n
  = \delta_{\alpha \beta}, \quad
  N_{\nu}= \dim H^0(K_M^{-\nu})-1.
  \end{align*}
  Note that this definition is independent of the choice of
  orthonormal basis of $H^0(K_M^{-\nu})$.
 \end{definition}

 \begin{definition}
 $\{(M^n, g(t)), 0 \leq t < \infty \}$ is called a tamed flow if there is a
 big integer $\nu$ such that the following properties hold.
 \begin{itemize}
    \item   $K_M^{-\nu}$ is very ample.
    \item   $\displaystyle   \snorm{F_{\nu}}{C^0(M \times [0, \infty))} <
    \infty$.
 \end{itemize}
 \label{definition: nc}
 \end{definition}

  Suppose $\{(M^n, g_t), 0 \leq t < \infty\}$ is a  tamed flow.
  Under the metric $g_t$ and $h_t^{\nu}$, we choose $\{ S_{ \nu, \beta}^{t}\}_{\beta=0}^N$ as an othonormal
  basis of $H^0(K_M^{-\nu})$.  At the same time, let $\{\tilde{S}_{\nu, \beta}^{t} \}_{\beta=0}^N$
  be an orthonormal basis of $H^0(K_M^{-\nu})$ under the metric $g_0$ and $h_0^{\nu}$.
  Then we have two embeddings.
  \begin{align*}
     \Phi^t: M \mapsto \CP^N, & \quad  x \mapsto [S_{\nu,0}^{t}(x):
     \cdots:  S_{\nu, N}^{t}(x)]; \\
     \Psi^t: M \mapsto \CP^N, & \quad  x \mapsto [\tilde{S}_{\nu,0}^t(x):
     \cdots:
     \tilde{S}_{\nu, N}^t(x)].
  \end{align*}
  By rotating basis if necessary, we can assume $\Phi^t=\sigma(t) \circ
  \Psi^t$ where
  \begin{align*}
   \sigma(t)= a(t) diag \{ \lambda_0(t), \cdots, \lambda_N(t)\} ,
   \quad 0<a(t),
   \quad  0 < \lambda_0(t) < \lambda_1(t) < \cdots < \lambda_N(t) =1.
  \end{align*}
  This indicates that \KRf equation can be rewritten as
  \begin{align*}
     \dot{\varphi} &= \log \frac{\omega_{\varphi}^n}{ \omega^n} +
     \varphi +u_{\omega}\\
      &= \frac{1}{\nu} \log \frac{\sum_{\beta=0}^{N} \snorm{S_{\nu, \beta}^t}{h_t^{\nu}}^2}{\sum_{\beta=0}^N \snorm{S_{\nu,
      \beta}^t}{h_0^{\nu}}^2} + \varphi +u_{\omega}\\
      &=\frac{1}{\nu}\log \sum_{\beta=0}^{N} \snorm{S_{\nu, \beta}^t}{h_t^{\nu}}^2 -
       \frac{1}{\nu}\log \sum_{\beta=0}^{N} \snorm{a(t) \lambda_{\beta}(t) \tilde{S}_{\nu,
      \beta}^t}{h_0^{\nu}}^2 + \varphi +u_{\omega}\\
      &=F_{\nu}(x, t)
      - \frac{1}{\nu} \log \sum_{\beta=0}^{N} \snorm{\lambda_{\beta}(t) \tilde{S}_{\nu,
      \beta}^t}{h_0^{\nu}}^2 + \varphi +u_{\omega} - \frac{2}{\nu} \log a(t). \\
  \end{align*}
 In other words,
 \begin{align*}
     \varphi - \frac{2}{\nu} \log a(t) = \dot{\varphi} -u_{\omega}
      - F_{\nu}(x,t)
      + \frac{1}{\nu} \log \sum_{\beta=0}^{N} \snorm{\lambda_{\beta}(t) \tilde{S}_{\nu, \beta}^t}{h_0^{\nu}}^2.
 \end{align*}
 Since $\displaystyle F_{\nu}(x,t)$,
 $\dot{\varphi}$ and $u_{\omega}$ are all uniformly bounded, we obtain
 \begin{align*}
    \varphi - \frac{2}{\nu} \log a(t) \sim  \frac{1}{\nu} \log \sum_{\beta=0}^{N} \snorm{\lambda_{\beta}(t) \tilde{S}_{\nu,
      \beta}^t}{h_0^{\nu}}^2.
 \end{align*}
 Here we use the notation $\sim$ to denote that the difference of two
 sides are controlled by a constant.    It follows that
 \begin{align*}
     \varphi - \sup_M \varphi  \sim \frac{1}{\nu} \log \sum_{\beta=0}^{N} \snorm{\lambda_{\beta}(t) \tilde{S}_{\nu,
      \beta}^t}{h_0^{\nu}}^2 - \frac{1}{\nu} \sup_M \log \sum_{\beta=0}^{N} \snorm{\lambda_{\beta}(t) \tilde{S}_{\nu,
      \beta}^t}{h_0^{\nu}}^2.
 \end{align*}
 It is obvious that
 \begin{align*}
    \sup_M \log \sum_{\beta=0}^{N} \snorm{\lambda_{\beta}(t) \tilde{S}_{\nu,
      \beta}^t}{h_0^{\nu}}^2 \leq \sup_M \log \sum_{\beta=0}^{N} \snorm{ \tilde{S}_{\nu,
      \beta}^t}{h_0^{\nu}}^2 < C.
 \end{align*}
 On the other hand, we have
 \begin{align*}
   \sup_M \log \sum_{\beta=0}^{N} \snorm{\lambda_{\beta}(t) \tilde{S}_{\nu,
      \beta}^t}{h_0^{\nu}}^2 & \geq \sup_M \log \snorm{\lambda_{N}(t) \tilde{S}_{\nu,
      \beta}^t}{h_0^{\nu}}^2 \\
      &= \sup_M \log  \snorm{\tilde{S}_{\nu, N}^t}{h_0^{\nu}}^2\\
      &=\log \sup_M \snorm{\tilde{S}_{\nu, N}^t}{h_0^{\nu}}^2\\
      &\geq \log \frac{1}{V} \int_M \snorm{\tilde{S}_{\nu, N}^t}{h_0^{\nu}}^2
      \omega^n\\
      &=-\log V.
 \end{align*}
 Therefore, $\displaystyle  \frac{1}{\nu} \sup_M \log \sum_{\beta=0}^{N} \snorm{\lambda_{\beta}(t) \tilde{S}_{\nu,
 \beta}^t}{h_0^{\nu}}^2$ is uniformly bounded and it yields that
 \begin{align}
    \varphi - \sup_M \varphi \sim  \frac{1}{\nu} \log \sum_{\beta=0}^{N} \snorm{\lambda_{\beta}(t) \tilde{S}_{\nu,
    \beta}^t}{h_0^{\nu}}^2.
 \label{eqn: pstrongprev}
 \end{align}
 So we have proved the following property.
 \begin{proposition}
 If $\{(M^n, g(t)), 0 \leq t < \infty \}$ is a \KRf tamed by $\nu$,
 then there is a constant $C$ (depending on this flow and $\nu$) such that
 \begin{align}
     |\varphi - \sup_M \varphi
     -\frac{1}{\nu}\log \sum_{\beta=0}^{N} \snorm{\lambda_{\beta}(t) \tilde{S}_{\nu, \beta}^t}{h_0^{\nu}}^2 | < C
 \label{eqn: strongC0e}
 \end{align}
 uniformly along this flow.
 \label{proposition: pstrong}
 \end{proposition}

 Inequality (\ref{eqn: strongC0e}) is called the strong partial $C^0$-estimate by Tian.
 Using this estimate, the control of $\norm{\varphi}{C^0(M)}$ along \KRf
 is reduced to the control of values of local $\alpha$-invariants (See Definition~\ref{definitionin: lalpha}
 and Definition~\ref{definitionin: nualpha})
  of holomorphic sections $S \in H^0(K_M^{-\nu})$.\\

 \begin{theorem}
  Suppose $\{(M^n, g(t)), 0 \leq t < \infty \}$ is a \KRf tamed by $\nu$.
  If $\alpha_{\nu, 1}> \frac{n}{n+1}$,
  then $\varphi$ is uniformly bounded along this flow. In
  particular, this flow converges to a KE metric exponentially fast.
 \label{theorem: nuconv}
 \end{theorem}

 \begin{proof}
   Suppose not. Then there is a sequence of times $t_i$ such that
   $\displaystyle \lim_{i \to \infty} \snorm{\varphi_{t_i}}{C^0(M)}= \infty$.

  Choose $S_{\nu, \beta}^t = a(t)\lambda_{\beta}(t) \tilde{S}_{\nu, \beta}^t, \quad 0 \leq \beta \leq N$ as
  before.   Since both $\snorm{\tilde{S}_{\nu, \beta}^t}{h_0^{\nu}}$ and $\lambda_{\beta}(t)$ are
  uniformly bounded, we can assume
   \begin{align*}
     \lim_{i \to \infty} \lambda_{\beta}(t_i) = \bar{\lambda}_{\beta}, \quad
     \lim_{i \to \infty} \tilde{S}_{\nu, \beta}^{t_i} = \bar{S}_{\nu, \beta},
         \quad \beta=0, 1, \cdots, N.
   \end{align*}
  Notice that $\bar{\lambda}_N =1$.

  Define
  $\displaystyle I(\alpha, t)
  = \int_M (\sum_{\beta=0}^N \snorm{\lambda_{\beta}(t)\tilde{S}_{\nu, \beta}}{h_0^{\nu}}^2)^{-\frac{\alpha}{\nu}} \omega^n$.
  Clearly, $I(\alpha, t_i) \leq \int_M \snorm{\tilde{S}_{\nu, N}}{h_0^{\nu}}^{- \frac{2\alpha}{\nu}} \omega^n$.
  As $\bar{S}_{\nu, N} \in H^0(K_M^{-\nu})$ and $\alpha_{\nu, 1}> \frac{n}{n+1}$,
  we can find a number $\alpha \in (\frac{n}{n+1}, \alpha_{\nu, 1})$
  such that  $\int_M |\bar{S}_{\nu, N}|_{h_0^{\nu}}^{-\frac{2\alpha}{\nu}} \omega^n < C$.
  By the semi continuity of singularity exponent, (c.f.~\cite{DK}, ~\cite{Tian90}),
  we have
  \begin{align*}
  \limsup_{i \to \infty} I(\alpha, t_i) \leq
   \lim_{i \to \infty} \int_M |\tilde{S}_{\nu, N}|_{h_0^{\nu}}^{-\frac{2\alpha}{\nu}}
     \omega^n  = \int_M |\bar{S}_{\nu, N}|_{h_0^{\nu}}^{-\frac{2\alpha}{\nu}} \omega^n <  C.
  \end{align*}
  Along a tamed flow, inequality
  $\displaystyle |\varphi - \sup_M \varphi
         - \frac{1}{\nu} \log \sum_{\beta=0}^{N}
          \snorm{\lambda_{\beta}(t) \tilde{S}_{\nu, \beta}^t}{h_0^{\nu}}^2| < C$
  holds.    It follows that  $\int_M e^{\alpha(\varphi_{t_i}-\sup_M \varphi_{t_i})} \omega^n <
  C$.
  Recall that $\dot{\varphi} = \log \frac{\omega_{\varphi}^n}{\omega^n} + \varphi +u_{\omega}$, we have
  \begin{align*}
     \frac{1}{V} \int_M e^{-\alpha(\varphi_{t_i} - \sup_M \varphi_{t_i})} \cdot
     e^{\varphi_{t_i} +u_{\omega} - \dot{\varphi}} \omega_{\varphi_{t_i}}^n < C.
  \end{align*}
 Note that both  $\dot{\varphi}$ and $u_{\omega}$  are uniformly
 bounded. It follows from the convexity of exponential map that
 $\alpha \sup_M \varphi_{t_i}+ (1- \alpha) \frac{1}{V}\int_M \varphi_{t_i} \omega_{\varphi_{t_i}}^n <
 C$. \;    In other words, it is
 \begin{align}
        \sup_M \varphi_{t_i} < \frac{1-\alpha}{\alpha}
  \frac{1}{V} \int_M (-\varphi_{t_i}) \omega_{\varphi_{t_i}}^n + C.
  \label{eqn: cofib}
 \end{align}
 Combining this with inequality (\ref{eqn: supphi}), we have
   \begin{align*}
      \sup_M \varphi_{t_i} < n\frac{1- \alpha}{\alpha} \sup_M \varphi_{t_i} + C.
   \end{align*}
 Since $\alpha \in (\frac{n}{n+1}, 1)$, it follows that
   $\displaystyle \sup_M \varphi_{t_i}$ is uniformly bounded from above.
   Consequently, $\varphi_{t_i}$ is uniformly bounded.
   This contradicts to our assumption for $\varphi_{t_i}$.
 \end{proof}

  By more careful analysis, we can improve this theorem a little
  bit.

 \begin{proposition}
   Let
    $\displaystyle X_t \triangleq \frac{1}{\nu} \log (\sum_{\beta=0}^{N} \snorm{\lambda_{\beta}(t) \tilde{S}_{\nu,
     \beta}^{t}}{h_0^{\nu}}^2)$.
   There is a constant $C$ such that
   \begin{align}
     \left| \frac{1}{V} \int_M \{\st \partial \varphi_{t_i} \wedge \bar{\partial}
      \varphi_{t_i}
        - \st \partial X_{t_i} \wedge \bar{\partial} X_{t_i} \} \wedge \omega^{n-1} \right| < C.
   \label{eqn: xphiequiv}
   \end{align}
 \end{proposition}

 \begin{proof}
   Since every $\tilde{S}_{\nu, \beta}^{t_i}$ is a holomorphic
   section, direct calculation shows that
   \begin{align*}
   \triangle X_{t_i} \geq -R
   \end{align*}
   where $\triangle$, $R$ are the Laplacian operator and
   the scalar curvature under the metric $\omega$.
   As $\triangle \varphi + n >0$, we can choose a constant
   $C_0$ such that  $\triangle \varphi + \triangle X_{t_i} + C_0>
   0$.    For the simplicity of
   notation, we omit the subindex $t_i$ in the following argument.
   So we have conditions
  \begin{align*}
       |\varphi -\sup_M \varphi - X|< C, \quad   \triangle X + \triangle \varphi + C_0>0.
  \end{align*}
   Having these conditions at hand, direct computation gives us
   \begin{align*}
     &\qquad \frac{1}{V}\int_M \{\st \partial \varphi \wedge \bar{\partial} \varphi
       - \st \partial X \wedge \bar{\partial} X \} \wedge \omega^{n-1}\\
   &=\frac{1}{V}\int_M (X-\varphi) (\triangle X + \triangle \varphi) \omega^n\\
   &=\frac{1}{V}\int_M (X-\varphi+ \sup_M \varphi +C) (\triangle X + \triangle \varphi)
   \omega^n\\
   &= \frac{1}{V}\int_M (X-\varphi + \sup_M \varphi +C) (\triangle X + \triangle \varphi + C_0)
   \omega^n  - \frac{C_0}{V} \int_M (X-\varphi+ \sup_M \varphi +C) \omega^n\\
   &\geq -\frac{C_0}{V} \int_M (X-\varphi+ \sup_M \varphi + C) \omega^n\\
   &\geq -2C_0C.
   \end{align*}
  On the other hand, similar calculation shows
    \begin{align*}
     &\qquad \frac{1}{V}\int_M \{\st \partial \varphi \wedge \bar{\partial} \varphi
       - \st \partial X \wedge \bar{\partial} X \} \wedge \omega^{n-1}\\
   &\leq -\frac{C_0}{V} \int_M (X-\varphi+ \sup_M \varphi- C) \omega^n\\
   &\leq 2C_0C.
   \end{align*}
 Consequently, we have
 \begin{align*}
   \left| \frac{1}{V}\int_M \{\st \partial \varphi \wedge \bar{\partial} \varphi
       - \st \partial X \wedge \bar{\partial} X \} \wedge
       \omega^{n-1} \right| \leq 2C_0C.
 \end{align*}
 \end{proof}

 It follows from this Proposition that
 inequality (\ref{eqn: dsupphi}) implies
  \begin{align}
   \frac{1}{V} \int_M (-\varphi) \omega_{\varphi}^n \leq
  n \sup_M \varphi  -  \frac{n-1}{V} \int_M \st \partial X \wedge
     \bar{\partial} X \wedge \omega^{n-1} + C.
  \label{eqn: dxsupphi}
 \end{align}

 Similar to the theorems in~\cite{Tian90}, we can prove the
 following theorem.

 \begin{theorem}
 Suppose $\{(M^n, g(t)), 0 \leq t < \infty \}$ is a \KRf tamed by
 $\nu$,  $M$ is a Fano manifold satisfying $\alpha_{\nu, 2}>\frac{n}{n+1}$
 and
 $\alpha_{\nu, 1} > \frac{1}{2- \frac{n-1}{(n+1)\alpha_{\nu,2}}}$.
 Then along this flow, $\varphi$ is
 uniformly bounded.  In particular, this flow converges to a KE
 metric exponentially fast.
 \label{theorem: nuconvr}
 \end{theorem}

 \begin{proof}

  Suppose not. We have a sequence of times $t_i$ such that
   $\displaystyle \lim_{i \to \infty} \snorm{\varphi_{t_i}}{C^0(M)}= \infty$.

  As before, we have
   \begin{align*}
     \lim_{i \to \infty} \tilde{S}_{\nu, \beta}^{t_i} = \bar{S}_{\nu, \beta},
         \quad \beta=0, 1, \cdots, N; \quad
     \lim_{i \to \infty} \lambda_{\beta}(t_i) = \bar{\lambda}_{\beta};
      \quad \bar{\lambda}_N =1.
   \end{align*}

  \setcounter{claim}{0}
  \begin{claim}
     $\bar{\lambda}_{N-1} =0$.
  \end{claim}

   Otherwise, $\bar{\lambda}_{N-1}>0$.  Fix some $\alpha \in (\frac{n}{n+1}, \alpha_{2,\nu})$,
  we calculate
    \begin{align*}
      I(\alpha, t_i) &= \int_M (\sum_{\beta=0}^N \snorm{\lambda_{\beta}(t)\tilde{S}_{\nu, \beta}^{t_i}}{h_0^{\nu}}^2)^{-\frac{\alpha}{\nu}}
      \omega^n\\
      &\leq \int_M ( \snorm{\lambda_{N-1}(t_i)\tilde{S}_{\nu, N-1}^{t_i}}{h_0^{\nu}}^2  + \snorm{\tilde{S}_{\nu, N}^{t_i}}{h_0^{\nu}}^2)^{-\frac{\alpha}{\nu}}
      \omega^n\\
      &\leq (\lambda_{N-1}(t_i))^{-2 \alpha} \int_M (\snorm{\tilde{S}_{\nu, N-1}^{t_i}}{h_0^{\nu}}^2
      + \snorm{\tilde{S}_{\nu, N}^{t_i}}{h_0^{\nu}}^2)^{-\frac{\alpha}{\nu}} \omega^n.\\
    \end{align*}
  For simplicity of notation, we look $h_0^{\nu}$ as the default
  metric on the line bundle $K_M^{-\nu}$ without writing it out
  explicitly.  Then semi continuity property implies
   \begin{align*}
       \lim_{i \to \infty} \int_M (\snorm{\tilde{S}_{\nu, N-1}^{t_i}}{}^2
      + \snorm{\tilde{S}_{\nu, N}^{t_i}}{}^2)^{-\frac{\alpha}{\nu}} \omega^n
      =\int_M (\snorm{\bar{S}_{\nu, N-1}}{}^2 + \snorm{\bar{S}_{\nu, N}}{}^2)^{-\frac{\alpha}{\nu}}
      \omega^n < \infty.
   \end{align*}
   It follows that
   \begin{align*}
      I(\alpha, t_i) < 2
      (\bar{\lambda}_{N-1})^{-\frac{2\alpha}{\nu}} \int_M (\snorm{\bar{S}_{\nu, N-1}}{}^2
      + \snorm{\bar{S}_{\nu, N}}{}^2)^{-\frac{\alpha}{\nu}}
      \omega^n < C_{\alpha}.
   \end{align*}
  Recall the definition of $I(\alpha, t_i)$, equation (\ref{eqn: strongC0e}) implies
  $ \displaystyle   \int_M e^{-\alpha(\varphi(t_i) - \sup_M \varphi(t_i))}   \omega^n < C$.
  Since $\alpha > \frac{n}{n+1}$, as we did in previous theorem,
  we will obtain the boundedness of $\snorm{\varphi_{t_i}}{C^0(M)}$.
  This contradicts to  the initial assumption of $\varphi_{t_i}$!
  Therefore $\bar{\lambda}_{N-1}=0$ and we finish the proof of this Claim 1. \\

   \begin{claim}
    For every small constant $\epsilon$, there is a constant $C$ such that
    \begin{align}
      (1-\epsilon) \alpha_{\nu, 2} \sup_M \varphi_{t_i} +
        (1- (1-\epsilon)\alpha_{\nu, 2}) \frac{1}{V} \int_M \varphi_{t_i}  \omega_{\varphi_{t_i}}^n
       \leq -2(1-\epsilon)\frac{\alpha_{\nu, 2}}{\nu} \log \lambda_{N-1}(t_i) +
       C.
      \label{eqn: a2e1}
    \end{align}
    \label{claim: a2e1}
  \end{claim}

  Fix $\epsilon$ small, we have
  \begin{align*}
   I((1-\epsilon)\alpha_{\nu, 2}, t_i) &=  \int_M (\sum_{\beta=0}^N  \snorm{\lambda_{\beta} \tilde{S}_{\nu,
      \beta}^{t_i}}{}^2)^{- \frac{(1-\epsilon)\alpha_{\nu, 2}}{\nu}} \omega^n\\
   &\leq  \int_M \{ \lambda_{N-1}(t_i)^2 [\snorm{\tilde{S}_{\nu, N-1}^{t_i}}{}^2
   + \snorm{\tilde{S}_{\nu, N}^{t_i}}{}^2]\}^{-\frac{(1-\epsilon)\alpha_{\nu, 2}}{\nu}}
   \omega^n\\
   & < C\lambda_{N-1}(t_i)^{-\frac{2(1-\epsilon)\alpha_{\nu, 2}}{\nu}}.
  \end{align*}
   The tamed condition implies that
  \begin{align*}
        \int_M e^{-(1-\epsilon)\alpha_{\nu, 2} (\varphi_{t_i} - \sup_M
        \varphi_{t_i})} \omega^n  < C \lambda_{N-1}(t_i)^{-\frac{2(1-\epsilon)\alpha_{\nu, 2}}{\nu}}.
  \end{align*}
  Plugging the equation
   $ \dot{\varphi} = \log \frac{\omega_{\varphi}^n}{\omega^n} + \varphi +u_{\omega}$
  into the previous inequality implies
  \begin{align*}
   C \lambda_{N-1}(t_i)^{-\frac{2(1-\epsilon)\alpha_{\nu, 2}}{\nu}}
    &> \int_M e^{-(1-\epsilon) \alpha_{\nu, 2}(\varphi_{t_i} -\sup_M \varphi_{t_i})}
     \cdot e^{\varphi_{t_i} +u_{\omega} - \dot{\varphi}_{t_i}}
     \omega_{\varphi_{t_i}}^n\\
    &=\int_M  e^{(1-\epsilon) \alpha_{\nu, 2} \sup_M \varphi_{t_i} + (1- (1-\epsilon)\alpha_{\nu, 2}) \varphi_{t_i}}
       \cdot e^{u_{\omega} - \dot{\varphi}_{t_i}}
       \omega_{\varphi_{t_i}}^n\\
    & \geq e^{-\norm{u_{\omega}}{C^0(M)} - \norm{\dot{\varphi}_{t_i}}{C^0(M)}}
        \int_M  e^{(1-\epsilon)\alpha_{\nu, 2} \sup_M \varphi_{t_i} + (1- (1-\epsilon)\alpha_{\nu, 2}) \varphi_{t_i}}
        \omega_{\varphi_{t_i}}^n\\
    &\geq e^{-\norm{u_{\omega}}{C^0(M)} - \norm{\dot{\varphi}_{t_i}}{C^0(M)}}
        \cdot V \cdot e^{(1-\epsilon)\alpha_{\nu, 2} \sup_M \varphi_{t_i} +
        (1- (1-\epsilon)\alpha_{\nu, 2}) \frac{1}{V} \int_M \varphi_{t_i}
        \omega_{\varphi_{t_i}}^n}
  \end{align*}
  Taking logarithm on both sides, we obtain inequality (\ref{eqn: a2e1}).
  This finishes the proof of Claim~\ref{claim: a2e1}.\\

  \begin{claim}
     For every small number $\epsilon>0$, there is a constant
     $C_{\epsilon}$ such that
  \begin{align}
      \frac{1}{V} \int_M \st \partial X_{t_i} \wedge
      \bar{\partial} X_{t_i} \wedge \omega^{n-1}
      \geq - \frac{(1-\epsilon)}{\nu}\log \lambda_{N-1}(t_i) - C_{\epsilon}.
   \label{eqn: a2e2}
  \end{align}
  \label{claim: a2e2}
  \end{claim}

  This proof is the same as the corresponding proof in
  \cite{Tian91}. So we omit it. \\

   Plugging inequality (\ref{eqn: a2e2}) into  inequality (\ref{eqn: dxsupphi}),
   together with inequality (\ref{eqn: a2e1}), we obtain
   \begin{align*}
   \left\{
   \begin{array}{ll}
   &\frac{1}{V} \int_M (-\varphi_{t_i}) \omega_{\varphi_{t_i}}^n \leq
   n \sup_M \varphi_{t_i}  + \frac{(n-1)(1-\epsilon)}{\nu}\log \lambda_{N-1}(t_i) +
   C,\\
   &\quad \quad\\
   &(1-\epsilon)\alpha_{\nu, 2} \sup_M \varphi_{t_i} +
        (1- (1-\epsilon)\alpha_{\nu, 2}) \frac{1}{V} \int_M \varphi_{t_i}  \omega_{\varphi_{t_i}}^n
       \leq -\frac{2(1-\epsilon)\alpha_{\nu, 2}}{\nu} \log \lambda_{N-1}(t_i) +C.
   \end{array}
   \right.
   \end{align*}
   Eliminating $\log \lambda_{N-1}(t_i)$, we have
   \begin{align*}
    \frac{1}{V} \int_M (-\varphi_{t_i}) \omega_{\varphi_{t_i}}^n
    \leq  \frac{(n+1) + (n-1)\epsilon}{((n+1)-(n-1)\epsilon)\alpha_{\nu, 2}
    -(n-1)} \alpha_{\nu, 2}  \sup_M \varphi_{t_i} + C.
   \end{align*}
  As inequality (\ref{eqn: cofib}), we have $\sup_M \varphi_{t_i} \leq
  \frac{1-(1-\epsilon)\alpha_{\nu, 1}}{(1-\epsilon)\alpha_{\nu, 1}} \frac{1}{V} \int_M (-\varphi_{t_i}) \omega_{\varphi_{t_i}}^n +
  C$.  It follows that
  \begin{align*}
     \{1 - \frac{(n+1) + (n-1)\epsilon}{((n+1)-(n-1)\epsilon)\alpha_{\nu, 2}
    -(n-1)} \cdot \alpha_{\nu, 2}  \cdot \frac{1-(1-\epsilon)\nu \alpha_{\nu, 1}}{(1-\epsilon)\nu \alpha_{\nu, 1}} \}
    \frac{1}{V} \int_M (-\varphi_{t_i}) \omega_{\varphi_{t_i}}^n
    \leq   C
  \end{align*}
  for every small constant $\epsilon$ and some big constant
  $C$ depending on $\epsilon$.        Since we have
  $\alpha_{\nu, 1} > \frac{1}{2- \frac{n-1}{(n+1)\alpha_{\nu, 2}}}= \frac{A}{A+1}$
  where $A= \frac{(n+1)\alpha_{\nu, 2}}{(n+1)\alpha_{\nu, 2} - (n-1)}$, so we
  can choose $\epsilon$ small enough such that
  \begin{align*}
   1 - \frac{(n+1) + (n-1)\epsilon}{((n+1)-(n-1)\epsilon)\alpha_{\nu, 2}
    -(n-1)} \cdot \alpha_{\nu, 2}  \cdot \frac{1-(1-\epsilon)\nu \alpha_{\nu, 1}}{(1-\epsilon)\nu \alpha_{\nu, 1}}>0.
  \end{align*}
  This implies that $\frac{1}{V} \int_M (-\varphi_{t_i})
  \omega_{\varphi_{t_i}}^n$ is uniformly bounded.  Therefore, $\snorm{\varphi_{t_i}}{C^0(M)}$
  is uniformly bounded.   Contradiction!
  \end{proof}

  \begin{remark}
    The methods applied in Theorem~\ref{theorem: nuconv} and
    Theorem~\ref{theorem: nuconvr} originate from \cite{Tian91}.
  \end{remark}

  \section{Plurianticanonical Line Bundles  and Tamed Condition}
   In this section, we study the basic properties of
    normalized holomorphic section $S \in H^0(K_M^{-\nu})$
   under the evolving metric  $\omega_{\varphi_t}$ and $h_t^{\nu}$.

\subsection{Uniform Bounds for Plurianticanonical Holomorphic Sections}
   Let $S$ be a normalized holomorphic section of $H^0(M,
 K_M^{-\nu})$, i.e., $\int_M \snorm{S}{h_t^{\nu}}^2 \omega_{\varphi_t}^n=1$.
 In this section, we will show both $\norm{\snorm{S}{h_t^{\nu}}}{C^0}$ and
 $\norm{\snorm{\nabla S}{h_t^{\nu}}}{C^0}$ are uniformly bounded.

 \begin{lemma}
  $\{(M^n, g(t)), 0 \leq t < \infty \}$ is a \KRf solution. There is a constant $A_0$ depending only on this
  flow such that  $\snorm{S}{h_t^{\nu}} < A_0 \nu^{\frac{n}{2}}$  whenever $S \in H^0(M, K^{-\nu})$
  satisfies $\int_M \snorm{S}{h_t^{\nu}}^2 \omega_{\varphi_t}^n =1$.
 \label{lemma: sectionbound}
 \end{lemma}

 \begin{proof}
    Fix a time $t$ and do all the calculations under the metric
    $g_t$ and $h_t^{\nu}$. Recall we have uniform Soblev constant and weak \Poincare
    constant, so we can do analysis uniformly independent of time
    $t$.

 \begin{clm}
    $S$ satisfies the equation
     \begin{align}
         \triangle \snorm{S}{}^2 = \snorm{\nabla S}{}^2 - \nu R
         \snorm{S}{}^2.
     \label{eqn: lapS}
     \end{align}
 \end{clm}
  This calculation can be done locally. Fix a point $x \in M$. Let $U$ be a neighborhood of $x$ with
   coordinate $\{z^1, \cdots, z^n\}$.  Then $K_M^{-\nu}$ has a
   natural trivialization on the domain $U$ and we can write
  $ S = f (\D{}{z^1} \wedge \cdots \D{}{z^n})^{\nu} $ for some holomorphic function $f$
   locally. For convenience, we denote $h= \det g_{k\bar{l}}$. Therefore, direct
  calculation shows
  \begin{align*}
     \triangle \snorm{S}{}^2 &= g^{i\bar{j}} \{ f\bar{f} h^{\nu}
     \}_{i\bar{j}}\\
        &= g^{i\bar{j}} \{f_i \bar{f} h^{\nu} + \nu f \bar{f} h^{\nu-1} h_i
        \}_{\bar{j}}\\
        &=g^{i\bar{j}}\{ f_i \bar{f}_{\bar{j}} h^{\nu} + \nu \bar{f}h^{\nu-1} f_i h_{\bar{j}}
          + \nu f h^{\nu-1} \bar{f}_{\bar{j}}h_i  + \nu (\nu -1) f\bar{f} h^{\nu-2} h_i h_{\bar{j}}
               + \nu f \bar{f} h^{\nu-1} h_{i\bar{j}}\}.
  \end{align*}
  If we choose normal coordinate at the point $x$, then we have
  $h=1$, $h_i=h_{\bar{j}}=0$,  $h_{i\bar{j}}= - R_{i\bar{j}}$. Plugging
  them into previous equality we have
  \begin{align*}
          \triangle \snorm{S}{}^2 = g^{i\bar{j}}\{ f_i \bar{f}_{\bar{j}} - \nu f \bar{f} R_{i\bar{j}}\}
            = \snorm{\nabla S}{}^2 -\nu R \snorm{S}{}^2.
  \end{align*}
  So equation (\ref{eqn: lapS}) is proved.\\

 From equation (\ref{eqn: lapS}), we have
  \begin{align*}
     \int_M \snorm{\nabla S}{}^2 d\mu = \int_M \nu R \snorm{S}{}^2
     d\mu \leq \nu \mathcal{B}
  \end{align*}
 where $d\mu=\omega_{\varphi_t}^n$.
 Note that volume is fixed along \KRf solution, we can omit the volume term in Sobolev inequality by adjusting $C_S$.
 Therefore, Sobolev inequality implies
  \begin{align*}
     \{\int_M \snorm{S}{}^{\frac{2n}{n-1}} d\mu \}^{\frac{n-1}{n}}
      &\leq C_S \{ \int_M \snorm{S}{}^2 d\mu + \int_M \snorm{\nabla \snorm{S}{}}{}^2 d\mu \}\\
      &\leq C_S \{ \int_M \snorm{S}{}^2 d\mu + \int_M \snorm{\nabla S}{}^2 d\mu \}\\
      &\leq C_S \{1+ \nu \mathcal{B}\} < C\nu.
  \end{align*}
  Here we use the property that $\bar{\nabla} S=0$.

  Note that we have the inequality $\triangle \snorm{S}{}^2 \geq -\nu R
  \snorm{S}{}^2$.  Let $u=|S|^2$, we have
  \begin{align*}
     \triangle u  \geq - \nu \mathcal{B} u,
     \quad \norm{u}{L^{\frac{n}{n-1}}} < C\nu^{\frac12}
  \end{align*}
   Multiplying this inequality by $u^{\beta -1} (\beta>1)$ and integration by parts implies
  \begin{align*}
   \int_M |\nabla u^{\frac{\beta}{2}}|^2 d\mu \leq \frac{\beta^2}{4(\beta -1)}
   \cdot (\mathcal{B}\nu) \cdot \int_M u^{\beta} d\mu.
  \end{align*}
  Combining this with Sobolev inequality yields
  \begin{align*}
      \{\int_M u^{\frac{n\beta}{n-1}} d\mu\}^{\frac{n-1}{n}} \leq
      C_S (1 + \frac{\beta^2 \mathcal{B} \nu}{4(\beta -1)}) \int_M u^{\beta}d\mu
      \leq C \nu \beta \int_M u^{\beta} d\mu.
  \end{align*}
  It follows that $\norm{u}{L^{\frac{n\beta}{n-1}}} \leq
  (C\nu)^{\frac{1}{\beta}} \beta^{\frac{1}{\beta}}
  \norm{u}{L^{\beta}}$.  Let $\beta = (\frac{n}{n-1})^k$, we have
  \begin{align*}
   \norm{u}{L^{\infty}} \leq (C \nu)^{\sum_{k=1}^{\infty}
   (\frac{n-1}{n})^k}  \cdot (\frac{n}{n-1})^{\sum_{k=1}^{\infty} k(\frac{n-1}{n})^k} \cdot
   \norm{u}{L^{\frac{n}{n-1}}}
   \leq C \nu^{\sum_{k=0}^{\infty}
   (\frac{n-1}{n})^k}
   =C\nu^{n}.
  \end{align*}
 In other words, $\norm{|S|}{L^{\infty}} \leq C\nu^{\frac{n}{2}}$.
 Let $A_0$ be the last $C$, we finish the proof.
 \end{proof}

 \begin{corollary}
     $\{(M^n, g(t)), 0 \leq t < \infty \}$ is a \KRf solution.  Then
     we have
 \begin{align}
     F_{\nu}(x, t) \leq \frac{2 \log A_0 + n\log \nu}{\nu} < B_0,
     \quad \forall \; x \in M, \; t \in [0, \infty), \; \nu \geq 1.
 \label{eqn: unifbound}
 \end{align}
 Here $B_0$ is a constant depending only on $A_0$.
 \label{corollary: unifbound}
 \end{corollary}

 \begin{proof}
   According to the definition of $F_{\nu}$, we only need to show
 \begin{align*}
   \sum_{\beta=0}^{N_{\nu}} \snorm{S_{\nu, \beta}^t}{h_t^{\nu}}^2(x)
   \leq A_0^2 \nu^n
 \end{align*}
 for every orthonormal holomorphic section basis $\{S_{\nu,
 \beta}^t\}_{\beta=0}^{N_{\nu}}$.  However, fix $x$, by rotating basis,
 we can always find a basis such that
 \begin{align*}
   \snorm{S_{\nu, \beta}^t}{h_t^{\nu}}^2(x)=0, \quad  1 \leq \beta
   \leq N_{\nu}.
 \end{align*}
 Therefore, by Lemma~\ref{lemma: sectionbound}, we have
 \begin{align*}
    \sum_{\beta=0}^{N_{\nu}} \snorm{S_{\nu, \beta}^t}{h_t^{\nu}}^2(x)
    =\snorm{S_{\nu, 0}^t}{h_t^{\nu}}^2(x) \leq A_0^2
    \nu^n.
 \end{align*}
 \end{proof}

\begin{lemma}
    $\{(M^n, g(t)), 0 \leq t < \infty \}$ is a \KRf solution.
     There is a constant $A_1$ depending only on this
  flow  and $\nu$ such that  $\snorm{\nabla S}{h_t^{\nu}}< A_1$  whenever $S \in H^0(M, K^{-\nu})$
  satisfying   $\int_M \snorm{S}{h_t^{\nu}}^2 \omega_{\varphi_t}^n =1$.
 \label{lemma: gsbound}
 \end{lemma}

 \begin{proof}
   Fix a time $t$ and then do all the computations with respect to
   $g(t)$ and $h_t^{\nu}$.  Same as in the previous Lemma,
   we can do uniform analysis since the existence
   of uniform Sobolev and weak \Poincare constants.

 \setcounter{claim}{0}
 \begin{claim}
   $\snorm{\nabla S}{}^2$ satisfies the equation
  \begin{align}
    \triangle \snorm{\nabla S}{}^2
    &= \snorm{\nabla \nabla S}{}^2 + \nu^2 \snorm{Ric}{}^2\snorm{S}{}^2
          -\nu R \snorm{\nabla S}{}^2 \notag \\
  & \qquad -(2\nu-1) R_{i\bar{k}} S_k \bar{S}_{\bar{i}}
    -\nu \{S R_i \bar{S}_{\bar{i}} + \bar{S}  R_{\bar{i}}S_i\}.
   \label{eqn: nablas}
   \end{align}
 \label{claim: nablas}
 \end{claim}

       Suppose $U$ to be a local coordinate around point $x$.
     Locally, we can rewrite
     \begin{align*}
     S&=f(\D{}{z^1} \wedge \cdots  \D{}{z^n})^{\nu},\\
     \nabla S &= \{ f_i + \nu f (\log h)_i \} dz^i \otimes (\D{}{z^1} \wedge \cdots
     \D{}{z^n})^{\nu}
     \end{align*}
     where $h = \det g_{k\bar{l}}$. It follows that
     \begin{align*}
       \snorm{\nabla S}{}^2 = g^{i\bar{j}}h^{\nu}
       (f_i + \nu f(\log h)_i)(\bar{f}_{\bar{j}} + \nu \bar{f} (\log
       h)_{\bar{j}}).
     \end{align*}
     Choose normal coordinate at point $x$. So at point $x$, we have
     $g_{i\bar{j}} = \delta_{i\bar{j}}$, $h=1$,  $h_i=h_{\bar{i}}=(\log h)_i = (\log h)_{\bar{i}}=0$,
     $h_{i\bar{j}}= (\log h)_{i\bar{j}}=- R_{i\bar{j}}$,
     $(\log h)_{ij}=(\log h)_{\bar{i}\bar{j}}=0$.
       So we compute
     \begin{align*}
       \triangle \snorm{\nabla S}{}^2 &=
       g^{k\bar{l}} \{  g^{i\bar{j}}h^{\nu} (f_i + \nu f(\log h)_i)(\bar{f}_{\bar{j}} + \nu \bar{f} (\log h)_{\bar{j}})
       \}_{k\bar{l}}\\
       &=g^{k\bar{l}} \{ -g^{i\bar{p}}g^{q\bar{j}} \D{g_{p\bar{q}}}{z^k} h^{\nu}(f_i + \nu f(\log h)_i)(\bar{f}_{\bar{j}}
       + \nu \bar{f}(\log h)_{\bar{j}})\\
       &\qquad \quad +\nu g^{i\bar{j}}h^{\nu-1} h_k (f_i + \nu f (\log h)_i)(\bar{f}_{\bar{j}} + \nu \bar{f} (\log h)_{\bar{j}}) \\
       &\qquad \quad + g^{i\bar{j}}h^{\nu} (f_{ik} + \nu f(\log h)_{ik} + \nu f_k (\log h)_i) (\bar{f}_{\bar{j}} +
          \nu \bar{f} (\log h)_{\bar{j}}) \\
       &\qquad \quad + g^{i\bar{j}} h^{\nu} (f_i + \nu f (\log h)_i)
       \nu \bar{f}(\log  h)_{\bar{j}k}\}_{\bar{l}}\\
     &=R_{j\bar{i}k\bar{k}} f_i \bar{f}_{\bar{j}}+ \nu h_{k\bar{k}} f_i \bar{f}_{\bar{i}} \\
     &\qquad + \nu f(\log h)_{k\bar{k} i} \bar{f}_{\bar{i}} + \nu f_k (\log h)_{i\bar{k}}\bar{f}_{\bar{i}} + f_{ik} \bar{f}_{\bar{i}\bar{k}}\\
     &\qquad +\nu^2 f \bar{f} (\log h)_{i\bar{k}}(\log
     h)_{\bar{i}k} + \nu f_i \bar{f} (\log h)_{k\bar{k} \bar{i}} + \nu f_i \bar{f}_{\bar{k}} (\log h)_{k\bar{i}}\\
     &=R_{j\bar{i}}f_i \bar{f}_{\bar{j}} - \nu R  \snorm{\nabla
     f}{}^2 - \nu f R_i \bar{f}_{\bar{i}} - \nu R_{i\bar{k}} f_k \bar{f}_{\bar{i}} \\
     & \qquad +\snorm{\nabla \nabla
     f}{}^2 + \nu^2  \snorm{f}{}^2 \snorm{Ric}{}^2 - \nu \bar{f} f_i R_{\bar{i}} -\nu R_{i\bar{k}} f_k \bar{f}_{\bar{i}}\\
     &= \snorm{\nabla \nabla S}{}^2 + \nu^2 \snorm{S}{}^2 \snorm{Ric}{}^2 - \nu R
     \snorm{\nabla S}{}^2  + (1-2\nu)R_{j\bar{i}}S_{i}\bar{S}_{\bar{j}}
       -\nu (S R_i \bar{S}_{\bar{i}} + \bar{S} R_{\bar{i}}S_i).
     \end{align*}
   So we finish the proof of Claim~\ref{claim: nablas}.\\

   \begin{claim}
     $S$ satisfies the equation
     \begin{align}
      S_{,i\bar{j}} = - \nu S R_{i \bar{j}}
      \label{eqn: hessians}
     \end{align}
   \label{claim: hessians}
   \end{claim}

         Suppose $U$ to be a normal coordinate around point $x$.
     Locally, we can rewrite
     \begin{align*}
     S&=f(\D{}{z^1} \wedge \cdots  \D{}{z^n})^{\nu},\\
     \nabla S &= \{ f_i + \nu f (\log h)_i \} dz^i \otimes (\D{}{z^1} \wedge \cdots
     \D{}{z^n})^{\nu},\\
     \bar{\nabla } S &= f_{\bar{i}} d \bar{z^i} \otimes (\D{}{z^1} \wedge \cdots
     \D{}{z^n})^{\nu}=0.
     \end{align*}
   Recall $\Gamma_{ij}^k = g^{k\bar{l}} \D{g_{i{\bar{l}}}}{z^j}$, it vanishes  at point $x$.
   So $(\log h)_i$, $(\log h)_{ij}$ vanish at point  $x$.  Note that $f$ is holomorphic, and our
   connection is compatible both with the
   metric and the complex structure. So  $\bar{\nabla} \nabla S$ has
   only one term
   \begin{align*}
     \bar{\nabla} \nabla S &= \{ \nu f (\log h)_{i \bar{j}}\} d
     \bar{z^j} \otimes dz^i  \otimes (\D{}{z^1} \wedge \cdots
      \D{}{z^n})^{\nu}\\
     &=  -\nu fR_{i\bar{j}} d \bar{z^j} \otimes dz^i  \otimes (\D{}{z^1} \wedge \cdots
      \D{}{z^n})^{\nu}.
   \end{align*}
   It follows that $S_{,i\bar{j}} = - \nu S R_{i \bar{j}}$.
   Claim~\ref{claim: hessians} is proved.\\

   \begin{claim}
    There is a constant $C$ such that
    $\norm{\snorm{\nabla S}{}}{L^{\frac{2n}{n-1}}} < C$ uniformly.
    \label{claim: nablaSLp}
   \end{claim}

   Integrate both sides of equation (\ref{eqn: nablas}) and we have
   \begin{align*}
    \int_M \snorm{\nabla \nabla S}{}^2 d\mu &\leq \int_M \nu R \snorm{\nabla
    S}{}^2 d\mu + (2\nu-1) \int_M R_{i\bar{k}}S_k \bar{S}_{\bar{i}}
    d\mu  + \nu \int_M \{ SR_i\bar{S}_{\bar{i}} + \bar{S}R_{\bar{i}}S_i  \}d\mu
   \end{align*}
   where $d\mu= \omega_{\varphi_t}^n$. Recall that $R_{i\bar{k}} = g_{i\bar{k}} -
   \dot{\varphi}_{i\bar{k}}$. It follows that
   \begin{align*}
   \int_M \snorm{\nabla \nabla S}{}^2 d\mu
      &\leq \int_M \nu R \snorm{\nabla S}{}^2 d\mu+(2\nu-1) \int_M \snorm{\nabla S}{}^2
     d\mu - (2\nu-1) \int_M \dot{\varphi}_{i\bar{k}} S_k
     \bar{S}_{\bar{i}} d\mu\\
     &\qquad \qquad + 2\int_M \{-\nu R \snorm{\nabla S}{}^2 + \nu^2 R^2\snorm{S}{}^2 \}
     d\mu\\
     &=2 \nu^2 \int_M R^2 \snorm{S}{}^2 d\mu - \nu \int_M R |\nabla S|^2 d\mu\\
     &\qquad \qquad + (2\nu-1) \int_M
     \snorm{\nabla S}{}^2 d\mu +(2\nu-1) \int_M \dot{\varphi}_i\{S_{,k\bar{k}} \bar{S}_{\bar{i}}
       + S_k\bar{S}_{,\bar{i} \bar{k}} \} d\mu
   \end{align*}
   Note that we used the property $S_{,l \bar{k}}= -\nu S R_{l\bar{k}}$.

 It follows that
   \begin{align*}
    \int_M \snorm{\nabla \nabla S}{}^2 d\mu &\leq
     2 \nu^2 \int_M R^2 \snorm{S}{}^2 d\mu
     - \nu \int_M R |\nabla S|^2 d\mu + (2\nu-1) \int_M
     \snorm{\nabla S}{}^2 d\mu \\
     &\qquad   - (2\nu-1)\nu \int_M S R \dot{\varphi}_i
     \bar{S}_{\bar{i}} d\mu
     +(2\nu-1) \int_M \bar{S}_{,\bar{i} \bar{k}}
     \dot{\varphi}_i S_k d\mu\\
     &\qquad \textrm{(Recall that $R,\dot{\varphi}, |\nabla \dot{\varphi}|, \snorm{S}{}, \; \int_M \snorm{S}{}^2 d\mu$ and $ \int_M \snorm{\nabla S}{}^2
     d\mu$ are all bounded.)}\\
     &\leq C\{1+ \int_M \snorm{\nabla S}{}d\mu + \int_M \snorm{\nabla \nabla S}{} \snorm{\nabla S}{} d\mu
     \}\\
     &\qquad \textrm{(Using \Holder inequality and interperlation: $xy \leq x^2 +
     y^2$)}\\
     &\leq C \{ 1+ V + \int_M \snorm{\nabla S}{}^2 d\mu  + \frac{1}{2C} \int_M \snorm{\nabla \nabla S}{}^2 d\mu
        +  2C \int_M \snorm{\nabla S}{}^2 d\mu \}\\
     &=\frac12 \int_M \snorm{\nabla \nabla S}{}^2 d\mu + C \{1+V + (2C+1) \mathcal{B}\nu \}.
   \end{align*}
   By redefining $C$, we have proved that $\int_M \snorm{\nabla \nabla S}{}^2 d\mu < C$ uniformly.
   On the other hand, we know
   \begin{align*}
      \int_M \snorm{\bar{\nabla} \nabla S}{}^2 d\mu
       = \int_M \nu^2  \snorm{S}{}^2 \snorm{Ric}{}^2 d\mu < C \int_M
       \snorm{Ric}{}^2 d\mu < C.
   \end{align*}
   Therefore Sobolev inequality tells us that
   \begin{align*}
       (\int_M \snorm{\nabla S}{}^{\frac{2n}{n-1}} d\mu)^{\frac{n-1}{n}}
    & \leq C_S \{ \int_M \snorm{\nabla S}{}^2 d\mu + \int_M \snorm{\nabla \snorm{\nabla
    S}{}}{}^2  \} d\mu \\
     & \leq C \{ \int_M \snorm{\nabla S}{}^2 d\mu + \int_M \snorm{\nabla \nabla S}{}^2 d\mu
           + \int_M \snorm{\bar{\nabla} \nabla S}{}^2 \} d\mu.
   \end{align*}
   This means $\norm{\snorm{\nabla S}{}}{L^{\frac{2n}{n-1}}}$ is uniformly bounded along
   \KRfd  So we have finished the proof of the Claim~\ref{claim: nablaSLp}.\\

   Fix $\beta>1$, multiplying $-\snorm{\nabla S}{}^{2(\beta -1)}$ to both
   sides of equation (\ref{eqn: nablas}) and doing integration  yields
   \begin{align*}
    &\qquad \qquad \frac{4(\beta-1)}{\beta^2} \int_M \left|  \nabla \snorm{\nabla S}{}^{\beta}\right|^2
     d\mu\\
   &=  -\int_M (\nu^2 \snorm{Ric}{}^2\snorm{S}{}^2 +\snorm{\nabla \nabla S}{}^2)\snorm{\nabla
      S}{}^{2(\beta-1)} d\mu + \int_M \nu R \snorm{\nabla S}{}^{2\beta} d\mu\\
   & \qquad
   + \underbrace{\int_M (2\nu-1) R_{i\bar{k}}S_k \bar{S}_{\bar{i}} \snorm{\nabla
   S}{}^{2(\beta-1)}d\mu}_{I}
   + \underbrace{\nu \int_M  \{S R_i \bar{S}_{\bar{i}} + \bar{S} R_{\bar{i}}S_i\} \snorm{\nabla
   S}{}^{2(\beta-1)}d\mu}_{II}.
   \end{align*}
  Plugging  $R_{i\bar{k}}= g_{i\bar{k}} - {\dot{\varphi}}_{i\bar{k}}$ into $I$
  yields
  \begin{align*}
   I= (2\nu-1) \int_M \snorm{\nabla S}{}^{2\beta} d\mu - (2\nu-1)
   \int_M \dot{\varphi}_{i\bar{k}} S_k \bar{S}_{\bar{i}} \snorm{\nabla
   S}{}^{2(\beta-1)} d\mu
  \end{align*}
  Since $S_{,l\bar{k}}= -\nu S R_{l\bar{k}}$, using the uniformly
  boundedness
  of $\dot{\varphi}$,$R$ and $\snorm{S}{}$, we have
  \begin{align*}
      \frac{I}{2\nu-1} &= \int_M \snorm{\nabla S}{}^{2\beta} d\mu
                     - \int_M  \dot{\varphi}_i (\nu R S   \bar{S}_{\bar{i}}- S_k \bar{S}_{,\bar{i} \bar{k}})
     \snorm{\nabla S}{}^{2(\beta-1)} d\mu\\
    &\qquad + (\beta-1) \int_M \dot{\varphi}_i S_k \bar{S}_{\bar{i}}\snorm{\nabla S}{}^{2(\beta-2)}(-\nu S R_{l\bar{k}} \bar{S}_{\bar{l}}
         + S_l \bar{S}_{,\bar{l}\bar{k}})  d\mu\\
    &\leq  \int_M \snorm{\nabla S}{}^{2\beta} d\mu
       + C\nu \int_M \snorm{\nabla S}{}^{2\beta-1} d\mu
       + C\int_M \snorm{\nabla \nabla S}{}
       \snorm{\nabla S}{}^{2\beta-1} d\mu\\
     &\qquad + C (\beta-1) \nu \int_M \snorm{Ric}{} \snorm{S}{} \snorm{\nabla S}{}^{2\beta-1} d\mu
       + C(\beta-1) \int_M \snorm{\nabla \nabla S}{} \snorm{\nabla
       S}{}^{2\beta-1}.
  \end{align*}
  Therefore, for some constant $C$ (It may depends on $\nu$), we have
  \begin{align*}
     I \leq C \int_M (\snorm{\nabla S}{}^{2\beta-1} + \snorm{\nabla S}{}^{2\beta}) d\mu
     + \beta C \int_M (\nu \snorm{Ric}{} \snorm{S}{}
         + \snorm{\nabla \nabla S}{}) \snorm{\nabla S}{}^{2\beta-1} d\mu.
  \end{align*}

  Direct calculation shows
  \begin{align*}
   II &= \nu \int_M \{-R \snorm{\nabla S}{}^{2\beta} + \nu \snorm{S}{}^2 R^2 \snorm{\nabla S}{}^{2(\beta-1)}
    \}d\mu\\
    &\quad - \nu\int_M  RS\bar{S}_{\bar{i}} (\beta-1) \snorm{\nabla S}{}^{2(\beta-2)}(-\nu
    \bar{S}  R_{i\bar{l}}S_l   + \bar{S}_{\bar{l}} S_{,li}) d\mu\\
    &\quad + \nu \int_M \{ -R\snorm{\nabla S}{}^{2\beta} + \nu \snorm{S}{}^2 R^2 \snorm{\nabla
    S}{}^{2(\beta-1)}\}du\\
    &\quad - \nu \int_M R\bar{S} S_i (\beta-1) \snorm{\nabla S}{}^{2(\beta-2)}(-\nu S R_{l\bar{i}}\bar{S}_{\bar{l}}
      + S_l \bar{S}_{,\bar{l} \bar{i}})d\mu\\
    &=-2\nu \int_M R \snorm{\nabla S}{}^{2\beta} d\mu + 2\nu^2 \int_M
     \snorm{S}{}^2R^2 \snorm{\nabla S}{}^{2(\beta-1)} d\mu\\
    & \quad + 2(\beta-1)\nu^2 \int_M R \snorm{S}{}^2 R_{i\bar{l}}S_l
     \bar{S}_{\bar{i}} \snorm{\nabla S}{}^{2(\beta-2)} d\mu\\
    &\quad
     -(\beta-1)\nu \int_M \{S_{,li}\bar{S}_{\bar{i}}
     \bar{S}_{\bar{l}}S
      + \bar{S}_{,\bar{l}\bar{i}} S_l S_i \bar{S} \} R \snorm{\nabla
      S}{}^{2(\beta-2)} d\mu\\
    &\leq C \int_M \{\snorm{\nabla S}{}^{2\beta} + \snorm{\nabla S}{}^{2(\beta-1)} \}d\mu
      + \beta C \int_M  (\nu \snorm{Ric}{} \snorm{S}{} +
      \snorm{\nabla \nabla S}{}) \snorm{\nabla S}{}^{2(\beta-1)} d\mu.
  \end{align*}
 Combining this estimate with the estimate of $I$ we have
 \begin{align}
 &\qquad \qquad \frac{4(\beta-1)}{\beta^2} \int_M \left|  \nabla \snorm{\nabla S}{}^{\beta}\right|^2
     d\mu \notag \\
   & \leq  -\int_M (\nu^2 \snorm{Ric}{}^2\snorm{S}{}^2 +\snorm{\nabla \nabla S}{}^2)\snorm{\nabla
      S}{}^{2(\beta-1)} d\mu \notag \\
   &\qquad +C \int_M \{\snorm{\nabla S}{}^{2\beta} + \snorm{\nabla S}{}^{2(\beta-1)}
   \}d\mu \notag \\
   &\qquad + \beta C \int_M  (\nu \snorm{Ric}{} \snorm{S}{}
     + \snorm{\nabla \nabla S}{})\{ \snorm{\nabla S}{}^{2(\beta-1)} + \snorm{\nabla S}{}^{2\beta-1}
     \}d\mu \label{eqn: bamean}
 \end{align}
 Since $\beta C \nu \snorm{Ric}{} \snorm{S}{} \snorm{\nabla S}{}^{2(\beta-1)}
 =(\nu \snorm{Ric}{} \snorm{S}{} \snorm{\nabla S}{}^{(\beta-1)}) \cdot (\beta C \snorm{\nabla
 S}{}^{(\beta-1)})$,  we see
 \begin{align*}
   \int_M \beta C \nu \snorm{Ric}{} \snorm{S}{} \snorm{\nabla
   S}{}^{2(\beta-1)} d\mu \leq \int_M \frac12 \nu^2 \snorm{Ric}{}^2
   \snorm{S}{}^2 \snorm{\nabla S}{}^{2(\beta-1)} d\mu +  \int_M  \frac12 (\beta C)^2 \snorm{\nabla
   S}{}^{2(\beta-1)}
   d\mu
 \end{align*}
 Similar deduction yields
 \begin{align*}
  &\qquad \beta C \int_M  (\nu \snorm{Ric}{} \snorm{S}{}
     + \snorm{\nabla \nabla S}{})\{ \snorm{\nabla S}{}^{2(\beta-1)} + \snorm{\nabla S}{}^{2\beta-1}
     \}d\mu\\
  \leq& \int_M (\nu^2 \snorm{Ric}{}^2\snorm{S}{}^2
  +\snorm{\nabla \nabla S}{}^2)\snorm{\nabla S}{}^{2(\beta-1)} d\mu +
  \beta^2 C^2  \int_M \{ \snorm{\nabla S}{}^{2(\beta-1)} +
   \snorm{\nabla S}{}^{2\beta}\}d\mu.
 \end{align*}
 By adjusting constant $C$, it follows from (\ref{eqn: bamean}) that
 \begin{align*}
  \frac{4(\beta-1)}{\beta^2} \int_M \left|  \nabla \snorm{\nabla S}{}^{\beta}\right|^2
     d\mu \leq  \beta^2 C^2  \int_M \{ \snorm{\nabla S}{}^{2(\beta-1)} +
   \snorm{\nabla S}{}^{2\beta}\}d\mu.
 \end{align*}
 If $\beta \geq \frac{n}{n-1}$, we have
 \begin{align*}
   \int_M \left|  \nabla \snorm{\nabla S}{}^{\beta}\right|^2 d\mu
   \leq (C\beta)^3 \int_M \{ \snorm{\nabla S}{}^{2(\beta-1)} + \snorm{\nabla S}{}^{2\beta}\}d\mu.
 \end{align*}
 Sobolev inequality tells us that
 \begin{align}
    ( \int_M \snorm{\nabla S}{}^{\beta \cdot \frac{2n}{n-1}})^{\frac{n-1}{n}}
    &\leq C_S\{ \int_M \snorm{\nabla S}{}^{2\beta} d\mu + \int_M \left|  \nabla \snorm{\nabla S}{}^{\beta}\right|^2 d\mu
    \} \notag\\
    & \leq (2C\beta)^3 \int_M \{ \snorm{\nabla S}{}^{2(\beta-1)} + \snorm{\nabla S}{}^{2\beta}\}d\mu.
    \label{eqn: iterbase}
 \end{align}
 From this inequality and the fact $\norm{|\nabla S|}{L^{\frac{2n}{n-1}}}$ is uniformly bounded,
 standard Moser iteration technique tells us
 $\norm{|\nabla S|}{L^{\infty}}< A_1$ for some uniform constant $A_1$.

 \end{proof}

\subsection{Convergence of Plurianticanonical Holomorphic Sections}

 In this subsection we use $L^2$-estimate for
 $\bar{\partial}$-operator  to study the convergence of
 plurianticanonical bundles.
  This section is very  similar to Section 5 of  Tian's paper~\cite{Tian90}.
  For the readers' and ourselves' convenience, we
  write down the arguments in detail.

  First let's list the important $\bar{\partial}$-lemma without
  proof.
  \begin{proposition}[c.f.\cite{Tian90}, Proposition 5.1.]
      Suppose $(M^n, g, J)$ is a complete K\"ahler manifold, $\omega$ is metric form compatible with
      $g$ and $J$, $L$ is a line bundle on $M$ with the hermitian metric $h$, and  $\psi$ is a
      smooth function on $M$. If
      \begin{align*}
         Ric(h) + Ric(g) + \st \ddb \psi \geq c_0  \omega
      \end{align*}
      for some uniform positive number $c_0$ at every point.   Then for any
      smooth $L$-valued $(0,1)$-form $v$ on $M$ with $\bar{\partial} v
      =0$ and $\int_M \snorm{v}{}^2 d\mu_g$ finite, there exists a
      smooth $L$-valued function $u$ on $M$ such that $\bar{\partial}
      u=v$ and
      \begin{align*}
          \int_M \snorm{u}{}^2 e^{-\psi} d\mu_g \leq \frac{1}{c_0}
          \int_M \snorm{v}{}^2 e^{-\psi} d\mu_g
      \end{align*}
      where $\snorm{\cdot}{}$ is the norm induced by $h$ and $g$.
  \label{proposition: dbp}
  \end{proposition}

 In our application, we fix $M$ to be a Fano manifold,
 $L=K_M^{-\nu}$ for some integer $\nu$.

 This Proposition assures that the plurigenera  is a continuous
 function in a proper moduli space of complex varieties under
 Cheeger-Gromov topology.

  \begin{theorem}
    $(M_i, g_i, J_i)$ is a sequence of Fano manifolds satisfying
   \begin{enumerate}
   \item[(a).]  There is an a priori constant $\mathcal{B}$ such that
    \begin{align*}
        C_S((M_i, g_i)) + \norm{R_{g_i}}{C^0(M_i)} + \norm{u_i}{C^0(M_i)}  <
        \mathcal{B}.
    \end{align*}
    Here $C_S((M_i, g_i))$ is the Sobolev constant of $(M_i, g_i)$,
    $R_{g_i}$ is the scalar curvature, $-u_i$ is the normalized Ricci
    potential. In other words, it satisfies
   \begin{align*}
     Ric_{g_i}- \omega_{g_i} =-\st \partial \bar{\partial}u_i,  \quad \frac{1}{V_{g_i}} \int_{M_i} e^{-u_i} d\mu_{g_i}=1.
   \end{align*}

   \item[(b).]  There is a constant $K$ such that $K^{-1}r^{2n} \leq \Vol(B(x, r)) \leq Kr^{2n}$
   for every geodesic ball $B(x, r) \subset M_i$ satisfying $r \leq 1$.

   \item[(c).]  $\displaystyle  (M_i, g_i, J_i) \sconv (\hat{M}, \hat{g},
   \hat{J})$ where $(\hat{M}, \hat{g}, \hat{J})$ is a Q-Fano normal variety.
   \end{enumerate}
     Then for any fixed positive integer $\nu$,  we have
  \begin{enumerate}
   \item  If $S_i \in H^0(M_i, K_{M_i}^{-\nu})$ and $\int_{M_i} \snorm{S_i}{}^2
   d\mu_{g_i}=1$, then by taking subsequence if necessary, we have
   $\hat{S} \in H^0(\hat{M}, K_{\hat{M}}^{-\nu})$ such that
   \begin{align*}
      S_i \sconv \hat{S},
      \quad \int_{\hat{M}} \snorm{\hat{S}}{}^2 d\mu_{\hat{g}}=1.
   \end{align*}

   \item  If $\hat{S} \in H^0(\hat{M}, K_{\hat{M}}^{-\nu})$ and $\int_{\hat{M}} \snorm{\hat{S}}{}^2
   d\mu_{\hat{g}}=1$, then there is a subsequence  of holomorphic
   sections $S_i \in H^0(M_i, K_{M_i}^{-\nu})$ and $\int_{M_i} \snorm{S_i}{}^2
   d\mu_{g_i}=1$ such that $S_i \sconv \hat{S}$.
  \end{enumerate}
  \label{theorem: bundleconv}
  \end{theorem}

  \begin{proof}

  For simplicity, we let $\nu=1$. Let $\mathcal{P}$ be the singular
  set of $\hat{M}$. As $\hat{M}$ is normal variety, Hausdorff
  dimension of $\mathcal{P}$ is not greater than $2n-4$. In virtue of condition (b) and (c), volume
  converges as $M_i$ converge to  $\hat{M}$.  Consequently,
  $K^{-1}r^{2n} \leq \Vol(B(x, r)) \leq K r^{2n}$ holds for every
  geodesic ball $B(x, r) \subset \hat{M}$ satisfying $r \leq 1$.
  Therefore, by the fact that $\dim(\mathcal{P}) \leq 2n-4$,  the
  Hausdorff dimension definition and packing ball method implies that there
  is a constant $\mathcal{V}$ such that $\Vol(B(\mathcal{P}, r)) \leq \mathcal{V} r^4$
  whenever $r$ is small.
  Now we prove part 1 and part 2 respectively. \\

  \textit{Part1. ``$\Longrightarrow$"}

      According to the proof
  of Lemma~\ref{lemma: sectionbound}, we see there is an a priori
  bound $A_0$ such that  $\norm{\snorm{S_i}{}}{C^0(M_i)} < A_0$.

   Fix any small number $\delta$ and define $U_{\delta}= \hat{M} \backslash B(\mathcal{P}, \delta)$.
  By the  definition of smooth convergence,  there exists a sequence of
    diffeomorphisms $\phi_i :  U_{\delta} \to \phi_i(U_{\delta}) \subset M_i$
  satisfying the following properties
  \begin{enumerate}
  \item[(1)]  $\phi_i^* g_i  \sconv \hat{g}$ uniformly on $U_{\delta}$;
  \item[(2)]  $(\phi_i^{-1})_* \circ J_i \circ (\phi_i)_*  \sconv
  \hat{J}$ uniformly on $U_{\delta}$.
  \end{enumerate}
  For convenience, define $(\phi_i)^* S_i \triangleq (\phi_i^{-1})_{*} S_i$.
  Clearly, $((\phi_i)^* S_i)|_{U_{\delta}}$ is a section of
   $(T^{(1,0)}\hat{M} \oplus  T^{(0,1)}\hat{M})|_{U_{\delta}}$ where $T^{(1,0)}\hat{M}$ and $T^{(0,1)}\hat{M}$
  are divided by the complex structure $\hat{J}$.
   Note that $\norm{\snorm{((\phi_i)^* S_i)|_{U_{\delta}} }{}}{C^0(U_{\delta})} <
   A_0$ and  $((\phi_i)^* S_i)|_{U_{\delta}}$ is holomorphic under the complex structure
  $(\phi_i^{-1})_* \circ J_i \circ (\phi_i)_*$. By Cauchy's
  integration formula, all covariant derivatives of $((\phi_i^{-1})_*
  S_i)|_{U_{\delta}}$ with respect to $(\phi_i)^* g_i$ are uniformly
  bounded in the domain $U_{2\delta}$.   Therefore there must exist a
  limit section $\hat{S}_{2\delta} \in (T^{(1,0)}\hat{M} \oplus  T^{(0,1)}\hat{M})|_{U_{2 \delta}} $
  and
  $ \displaystyle     (\phi_i)^* S_i \sconv \hat{S}_{2\delta} \quad \textrm{on}
     \; U_{2\delta}$.
  This section $\hat{S}_{2\delta}$ is automatically holomorphic with
  respect to $\hat{J}$ since $(\phi_i^{-1})_* \circ J_i \circ (\phi_i)_*  \sconv
  \hat{J}$ on $U_{2\delta} \subset U_{\delta}$.

    As $(M_i, g_i, J_i) \sconv (\hat{M}, \hat{g}, \hat{J})$, we have
    $\displaystyle \lim_{i \to \infty} V_{g_i}(M_i \backslash \phi_i(U_{2\delta}))<2 \mathcal{V}(2\delta)^4
     =32 \mathcal{V}\delta^4$.
  It follows that
    \begin{align*}
     1 \geq   \int_{U_{2\delta}} \snorm{(\phi_i)^* S_i}{}^2 d\mu_{\phi_i^* g_i}
     &= \int_{\phi_i(U_{2\delta})} \snorm{S_i}{}^2
        d\mu_{g_i}\\
    &=\int_{M_i} \snorm{S_i}{}^2 d\mu_{g_i}
        -\int_{M_i \backslash \phi_i(U_{2\delta})} \snorm{S_i}{}^2
        d\mu_{g_i}\\
    &> 1- 32A_0^2 \mathcal{V} \delta^4.
    \end{align*}
  Therefore, for each $\delta$, there is a limit holomorphic
  section $\hat{S}_{2\delta} \in H^0(\hat{U}_{2\delta},
  K_{U_{2\delta}}^{-1})$ satisfying
  \begin{align*}
     \norm{\snorm{\hat{S}_{2\delta}}{}}{C^0(U_{2\delta})} \leq
     A_0, \quad
     1 \geq \int_{U_{2\delta}} \snorm{\hat{S}_{2\delta}}{}^2 d\mu_{\hat{g}} \geq 1 -
     32A_0^2 \mathcal{V} \delta^4.
  \end{align*}

 Let $\delta= \delta_k = 2^{-k} \to 0$ and then take diagonal
 sequence, we obtain a subsequence of  sections $(\phi_{i_k}^{-1})_*
 S_{i_k}|_{U_{2\delta_k}}$ satisfying
 \begin{align*}
   (\phi_{i_k}^{-1})_* S_{i_k}|_{K} \sconv \hat{S}|_{K}, \quad
   \forall \; \textrm{compact set}  \; K \subset
   \hat{M} \backslash \mathcal{P}.
 \end{align*}
 This exactly means that $(\phi_{i_k}^{-1})_* S_{i_k} \sconv \hat{S}$
 on $\hat{M} \backslash \mathcal{P}$.  As $\hat{M}$ is a $Q$-Fano normal variety,
 $\hat{S}$ can be naturally extended to a holomorphic section of
 $H^0(K_{\hat{M}}^{-1})$. Moreover, we have
 \begin{align*}
   \int_{\hat{M}} \snorm{\hat{S}}{}^2 d\mu_{}
   = \int_{\hat{M} \backslash \mathcal{P}} \snorm{\hat{S}}{}^2 d\mu_{}
   =1,
 \end{align*}
 where the metric on $K_{\hat{M}}^{-\nu}$ is naturally $(\det \hat{g})^{\nu}$.
 So we finish the proof of part 1.\\

  \textit{Part2. ``$\Longleftarrow$"}

   Fix two small positive numbers $r, \delta$ satisfying $r \gg 2\delta$.
    Define  function $\eta_{\delta}$ to be a cutoff function taking
  value $1$ on $U_{2\delta}$ and $0$ inside $B(\mathcal{P}, \delta)$.
   $\eta_{\delta}$ also satisfies $\snorm{\nabla \eta_{\delta}}{\hat{g}}<
  \frac{2}{\delta}$.

    Like before,  there exists a sequence of
  diffeomorphisms $\phi_i :  U_{\delta} \to \phi_i(U_{\delta}) \subset M_i$
  satisfying the following properties
  \begin{enumerate}
  \item[(1)]  $\phi_i^* g_i  \sconv \hat{g}$ uniformly on $U_{\delta}$;
  \item[(2)]  $(\phi_i^{-1})_* \circ J_i \circ (\phi_i)_*  \sconv
  \hat{J}$ uniformly on $U_{\delta}$.
  \end{enumerate}
   $\phi_{i*}(\eta_{\delta} \hat{S})$  can be looked as a smooth section of the  bundle
   $\Lambda^{n} (T^{(1,0)}M_i \oplus T^{(0,1)}M_i)$ by natural extension.
   Let $\pi_i$ be the projection from
   $\Lambda^{n} (T^{(1,0)}M_i  \oplus  T^{(0,1)}M_i)$ to $\Lambda^{n} T^{(1,0)}M_i$ and denote $V_{\delta, i} = \pi_i (\phi_{i*}(\eta_{\delta}
   \hat{S}))$.     The smooth
   convergence of complex structures implies that $V_{\delta,i}$ is an almost holomorphic section of $\Lambda^{n}
   T^{(1,0)}M_i$. In other words,
   \begin{align}
     \lim_{i \to \infty}  \sup_{\phi_i(U_{2\delta})} \snorm{\bar{\partial} V_{\delta, i}}{}
   = \lim_{i \to \infty}  \sup_{\phi_i(U_{2\delta})} \snorm{ \bar{\partial} (\pi_i
        (\phi_i^*(\eta_{\delta}\hat{S})))}{}  = 0.
   \label{eqn: bar0}
   \end{align}
   Here $\bar{\partial}$ is calculated under the complex structure $J_i$.

   Notice that $V(B(\mathcal{P}, \delta)) \leq \mathcal{V} \delta^4$ when $\delta$ small.
   Denote $\mathcal{A}= \norm{\snorm{\hat{S}}{}}{C^0(\hat{M})}$.
   Note that $\mathcal{A}$ depends on $\hat{M}$ and
   $\hat{S}$ itself.  We have

   \begin{align*}
    1 \geq
    \lim_{i \to \infty} \int_{M_i} \snorm{V_{\delta, i}}{}^2 d\mu_{g_i}
    = \lim_{i \to \infty} \int_{M_i} \snorm{\pi_i (\phi_{i*}(\eta_{\delta} \hat{S}))}{}^2 d\mu_{g_i}
    \geq 1 - 2\mathcal{A}^2 \mathcal{V}(2\delta)^4 = 1 - 32\mathcal{A}^2 \mathcal{V}
    \delta^4.
   \end{align*}
   Recall $V_{\delta, i}$ vanishes on $B(\mathcal{P}, \delta)$, so we have
   \begin{align*}
    \int_{M_i} \snorm{\bar{\partial} V_{\delta, i}}{}^2
     d\mu_{g_i}
    & = \int_{\phi_{i}(U_{2\delta})} \snorm{\bar{\partial} V_{\delta,i}}{}^2   d\mu_{g_i}
      +\int_{\phi_{i}(U_{\delta} \backslash U_{2\delta})} \snorm{\bar{\partial} V_{\delta, i} }{}^2
      d\mu_{g_i}.
   \end{align*}
   By virtue of inequality (\ref{eqn: bar0}) and the fact $\snorm{\nabla \eta_{\delta}}{\hat{g}} < \frac{2}{\delta}$,
   $\Vol(\phi_{i}(U_{\delta} \backslash U_{2\delta})) \leq 2\mathcal{V}
   (2\delta)^4$,  we obtain
   \begin{align*}
   \int_{M_i} \snorm{\bar{\partial} V_{\delta, i}}{}^2 d\mu_{g_i}
   \leq 1000\mathcal{A}^2 \mathcal{V} \delta^2
   \end{align*}
   for large $i$.

   Let $h_i$ be the  hermitian metric on $K_{M_i}^{-1}$ induced by
   $g_i$. Clearly, we have
   \begin{align*}
      Ric(h_i)  + Ric(g_i) + \st \ddb{(-2u_i)} = 2(Ric(g_i) - \st \ddb u_i) =   2 \omega_{g_i}.
   \end{align*}
   So we  are able to apply Proposition~\ref{proposition: dbp}
   and obtain a smooth section $W_{\delta, i}$ of $K_{M_i}^{-1}$ such that
   \begin{align}
   \left\{
   \begin{array}{ll}
   & \bar{\partial} W_{\delta, i} = \bar{\partial} V_{\delta, i}\\
   &  \int_{M_i} \snorm{W_{\delta, i}}{}^2  e^{2u_i} d\mu_{g_i}
   \leq \frac12 \int_{M_i} \snorm{\bar{\partial} V_{\delta, i}}{}^2 e^{2u_i}
   d\mu_{g_i} \leq \frac{e^{2\mathcal{B}}}{2} \int_{M_i} \snorm{\bar{\partial} V_{\delta, i}}{}^2
   d\mu_{g_i} < 500\mathcal{A}^2 \mathcal{V} e^{2\mathcal{B}} \delta^2.
   \end{array}
   \right.
   \label{eqn: WL2}
   \end{align}
   Triangle inequality implies
   \begin{align}
    1 +  \sqrt{500\mathcal{A}^2 \mathcal{V}e^{2\mathcal{B}} \delta^2} > (\int_{M_i} \snorm{V_{\delta, i} - W_{\delta, i}}{}^2 d\mu_{g_i})^{\frac12}
    >
  \sqrt{1 -32\mathcal{A}^2 \mathcal{V} \delta^4}-  \sqrt{500\mathcal{A}^2 \mathcal{V}e^{2\mathcal{B}} \delta^2}.
  \label{eqn: trinorm}
    \end{align}
  Therefore
  $\displaystyle
      S_{\delta, i} =
      \frac{V_{\delta, i} - W_{\delta, i}}{(\int_{M_i} \snorm{V_{\delta, i}
      - W_{\delta, i}}{}^2 d\mu_{g_i})^{\frac12}}  $
  is a well defined holomorphic section of  $K_{M_i}^{-1}$.

  Direct computation shows that $W_{\delta, i}$ satisfies the  elliptic equation:
    \begin{align}
       \triangle (\snorm{W_{\delta, i}}{}^2)& = \snorm{\nabla W_{\delta, i}}{}^2
           + |\bar{\nabla} W_{\delta, i}|^2 - R \snorm{W_{\delta,i}}{}^2 +
           2 Re <W_{\delta,i}, \bar{\partial}^* \bar{\partial}
           W_{\delta,i}>  \notag \\
       &=\snorm{\nabla W_{\delta, i}}{}^2
           + |\bar{\nabla} W_{\delta, i}|^2 - R \snorm{W_{\delta,i}}{}^2 +
           2 Re <W_{\delta,i}, \bar{\partial}^* \bar{\partial} V_{\delta,i}> \notag\\
      & \geq \snorm{\nabla W_{\delta, i}}{}^2
           + |\bar{\nabla} W_{\delta, i}|^2-(R+1) \snorm{W_{\delta,i}}{}^2
       -\snorm{\bar{\partial}^* \bar{\partial}V_{\delta,i}}{}^2 \notag\\
      & \geq \snorm{\nabla W_{\delta, i}}{}^2
           + |\bar{\nabla} W_{\delta, i}|^2-2\mathcal{B} \{ \snorm{W_{\delta,i}}{}^2 +
      \frac{1}{2\mathcal{B}} \snorm{\bar{\partial}^*
      \bar{\partial}V_{\delta,i}}{}^2\}.
      \label{eqn: wine}
    \end{align}
  All geometric quantities are computed under the metric $g_i$ and complex
    structure $J_i$.  Let $f=\snorm{W_{\delta, i}}{}^2 + \frac{1}{2\mathcal{B}} \sup_{\varphi_i(U_{\frac{r}{2}})}\snorm{\bar{\partial}^*
    \bar{\partial}V_{\delta,i}}{}^2$, on $\varphi_i(U_{\frac{r}{2}})$, we have
    \begin{align*}
     \triangle f  \geq -2\mathcal{B} f.
    \end{align*}
    Applying local Moser iteration in $\phi_i(U_{\frac{r}{2}})$, we obtain
    \begin{align*}
         \norm{f}{C^0(\varphi_i(U_{r}))}
         &\leq C'(r, \mathcal{B}, \mathcal{A})
         \norm{f}{L^{\frac{n}{n-1}}(\varphi_i(U_{\frac{r}{2}}))}\\
         &= C'(r, \mathcal{B}, \mathcal{A})
           \{\norm{\snorm{W_{\delta,i}}{}^2}{L^{\frac{n}{n-1}}(\varphi_i(U_{\frac{r}{2}}))}
           + \frac{1}{2\mathcal{B}} \sup_{\varphi_i(U_{\frac{r}{2}})}\snorm{\bar{\partial}^* \bar{\partial}V_{\delta,i}}{}^2
           \}.
    \end{align*}
    Since  $\displaystyle \sup_{\varphi_i(U_{\frac{r}{2}})}\snorm{\bar{\partial}^* \bar{\partial}V_{\delta,i}}{}^2$
    tends to $0$ uniformly,  it follows that
    \begin{align}
     \norm{\snorm{W_{\delta,i}}{}^2}{C^0(\varphi_i(U_{r}))}
     \leq C^{''}(r, \mathcal{B},
     \mathcal{A})\norm{\snorm{W_{\delta,i}}{}^2}{L^{\frac{n}{n-1}}(\varphi_i(U_{\frac{r}{2}}))}.
    \label{eqn: infp}
    \end{align}
    On the other hand,  inequality (\ref{eqn: wine}) can be written
    as
    \begin{align*}
     \snorm{\nabla W_{\delta, i}}{}^2 + |\bar{\nabla} W_{\delta, i}|^2
     \leq \triangle (\snorm{W_{\delta,i}}{}^2) + 2\mathcal{B}
     \snorm{W_{\delta,i}}{}^2 + \snorm{\bar{\partial}^*
     \bar{\partial}V_{\delta,i}}{}^2.
    \end{align*}
    Combining this inequality with Sobolev inequality, we can apply
    cutoff function on $\phi_i(U_{\frac{r}{4}} \backslash U_{\frac{r}{2}})$
    to obtain
    \begin{align*}
    \norm{\snorm{W_{\delta,i}}{}^2}{L^{\frac{n}{n-1}}(\varphi_i(U_{\frac{r}{2}}))}
    \leq
    C^{'''}(r, \mathcal{B}, \hat{M})\{\norm{\snorm{W_{\delta,i}}{}^2}{L^1
    (\varphi_i(U_{\frac{r}{4}}))}
    + \sup_{\varphi_i(U_{\frac{r}{4}})}\snorm{\bar{\partial}^* \bar{\partial}V_{\delta,i}}{}^2 \}.
    \end{align*}
  Together with inequality (\ref{eqn: infp}), the fact
  $\displaystyle \sup_{\varphi_i(U_{\frac{r}{4}})}\snorm{\bar{\partial}^* \bar{\partial}V_{\delta,i}}{}^2 \to 0$
  implies that
  \begin{align*}
     \norm{\snorm{W_{\delta,i}}{}^2}{C^0(\varphi_i(U_{r}))}
     &\leq C^{''''}(r, \mathcal{B}, \mathcal{A},
     \hat{M})\norm{\snorm{W_{\delta,i}}{}^2}{L^1(\varphi_i(U_{\frac{r}{4}}))}\\
     &\leq C^{''''}(r, \mathcal{B}, \mathcal{A},
     \hat{M})\norm{\snorm{W_{\delta,i}}{}^2}{L^1(M_i)}\\
     &\leq C(r, \mathcal{B}, \mathcal{A}, \mathcal{V}, \hat{M}) \delta^2.
  \end{align*}
  The last inequality follows from estimate (\ref{eqn: WL2}) and the
  fact $|u_i|<\mathcal{B}$.

   Fix $r, \delta$ and let $i \to \infty$, we have
   $ \displaystyle
      \lim_{i \to \infty} \varphi_i^{*}(S_{\delta, i})
       = \frac{\hat{S} + \hat{W}_r}{ \displaystyle \lim_{i \to \infty} ( \int_{M_i} \snorm{V_{\delta,i} - W_{\delta,
       i}}{}^2 d\mu_{g_i})^{\frac12}}
   $
   on domain $U_r$.  Here $\hat{W}_r$ is a holomorphic section of $H^0(U_r, K_{U_r}^{-1})$
   with $\norm{\snorm{\hat{W}_r}{}}{C^0(U_r)} \leq C \delta$.
   It follows from this and inequality (\ref{eqn: trinorm}) that
    $\displaystyle   \lim_{\delta \to 0} \lim_{i \to \infty} \varphi_i^*(S_{\delta, i}) = \hat{S}$
    on domain $U_r$.   Let $\delta_k = 2^{-k}$ and take diagonal sequence, we
   obtain  $ \displaystyle  \lim_{k \to \infty} \varphi_{i_k}^* (S_{2^{-k}, i_k})=\hat{S}$
    on $U_{r}$.
    Then let $r=2^{-l}$ and take diagonal sequence one
   more time, we obtain a sequence of holomorphic sections $S_l \triangleq  S_{2^{-k_l}, i_{k_l}}$
   such that
   \begin{align*}
       \lim_{l \to \infty} \varphi_l^*(S_l)
        = \hat{S}, \qquad \textrm{on}  \; \hat{M} \backslash \mathcal{P}.
   \end{align*}
   Since every $S_l$ is a
   holomorphic section (w.r.t $(\phi_l^{-1})_* \circ J_l \circ (\phi_l)_*$ ), Cauchy integration formula
   implies that this
   convergence is actually in $C^\infty$-topology.
 \end{proof}

 \subsection{Justification of Tamed Condition}

   In this section, we show when the \KRf is tamed.

   \begin{theorem}
   Suppose $\{(M^n, g(t)), 0 \leq t < \infty\}$ is a \KRf satisfying
   the following conditions.
   \begin{itemize}
   \item volume ratio bounded from above, i.e., there exists a
   constant $K$ such that
   \begin{align*}
   \Vol_{g(t)}(B_{g(t)}(x, r)) \leq Kr^{2n}
   \end{align*}
   for every geodesic ball $B_{g(t)}(x, r)$ satisfying $r \leq 1$.
   \item weak compactness, i.e., for every sequence $t_i \to \infty$, by
   passing to subsequence, we have
   \begin{align*}
      (M, g(t_i)) \sconv (\hat{M}, \hat{g}),
   \end{align*}
   where $(\hat{M}, \hat{g})$ is a Q-Fano normal variety.
   \end{itemize}
 Then this flow is tamed.
 \label{theorem: justtamed}
 \end{theorem}

   \begin{proof}
    Suppose this result is false. For every $p_i= i!$,  $F_{p_i}$ is
    an unbounded function on $M \times [0, \infty)$.  By
    Corollary~\ref{corollary: unifbound}, $F_{p_i}$ has no lower bound.  Therefore,
    there exists a point $(x_i, t_i)$ such that
    \begin{align}
      F_{p_i}(x_i, t_i) < -p_i.
      \label{eqn: fastdecrease}
    \end{align}
    By weak compactness, we can assume
    that
    \begin{align*}
       (M, g(t_i)) \sconv (\hat{M}, \hat{g}).
    \end{align*}
    Moreover, as $\hat{M}$ is a Q-Fano variety, we can assume
    $\displaystyle e^{\nu F_{\nu}(y)} =\sum_{\alpha=0}^{N_{\nu}} \snorm{S_{\nu, \alpha}(y)}{\hat{\omega}^{\nu}}^2>c_0$
    on $\hat{M}$. Applying Theorem~\ref{theorem: bundleconv}, we have
     \begin{align*}
        \lim_{i \to \infty} e^{\nu F_{\nu}(x_i, t_i)}> \frac12 c_0.
     \end{align*}
     It follows that there are holomorphic sections
     $S_{\nu}^{(t_i)} \in  H^0(K_M^{-\nu})$ satisfying
     \begin{align*}
      \int_M \snorm{S_{\nu}^{(t_i)}}{h_{t_i}^{\nu}}^2 \omega_{t_i}^n=1,
      \quad
      \snorm{S_{\nu}^{(t_i)}}{h_{t_i}^{\nu}}^2(x_i)= e^{\nu F_{\nu}(x_i, t_i)}> \frac12 c_0.
     \end{align*}
     According to Lemma~\ref{lemma: sectionbound}, we see there is a constant $C$
     depending only on this flow such that
      \begin{align*}
      \snorm{S_{\nu}^{(t_i)}}{h_{t_i}^{\nu}} < C \nu^{\frac{n}{2}}.
      \end{align*}
    So we have
    \begin{align*}
      A \triangleq \int_M \snorm{(S_{\nu}^{(t_i)})^k}{h_{t_i}^{k\nu}}^2 \omega_{t_i}^n
       < V C^{2k} \nu^{nk}.
    \end{align*}
    Therefore, $A^{-\frac12} (S_{\nu}^{(t_i)})^k$ are unit sections
    of $H^0(K_M^{-k\nu})$. It
    follows that
    \begin{align*}
     e^{k\nu F_{k\nu}(x_i, t_i)} \geq \snorm{A^{-\frac12}
      (S_{\nu}^{(t_i)})^k}{h_{t_i}^{k\nu}}^2(x_i)
      \geq V^{-1} C^{-2k}\nu^{-nk} \snorm{(S_{\nu}^{(t_i)})^k}{h_{t_i}^{k\nu}}^2(x_i)
      \geq V^{-1} C^{-2k}\nu^{-nk} (\frac{c_0}{2})^k.
    \end{align*}
    This implies that
    \begin{align*}
     k\nu \cdot F_{k\nu}(x_i, t_i) \geq -2k \log C - nk \log \nu + k \log(\frac{c_0}{2})
      -\log V
    \end{align*}
    for large $i$ (depending on $\nu$) and every $k$.   Let
    $k=\frac{p_i}{\nu}=\frac{i!}{\nu}$,
    by virtue of inequality (\ref{eqn: fastdecrease}), we have
    \begin{align*}
     -k^2 \nu^2 = -p_i^2 > p_i F_{p_i}(x_i, t_i) = k\nu \cdot F_{k\nu}(x_i, t_i)  \geq -2k \log C - nk \log \nu + k
     \log(\frac{c_0}{2}).
    \end{align*}
    However, this is impossible for large $k$!
  \end{proof}

  In Theorem 4.4 of~\cite{CW3}, we have proved the weak compactness
  property of \KRf on Fano surfaces, i.e., every sequence of
  evolving metrics of a \KRf solution on a Fano surface subconverges to
  a K\"ahler Ricci soliton orbifold in Cheeger-Gromov topology.
  Moreover, the volume ratio upper bound is proved as a lemma to
  prove weak compactness. As
  an application of this property, we obtain

 \begin{corollary}
   If $\{(M^2, g(t)), 0 \leq t < \infty\}$ is a \KRf on a Fano surface
   $M^2$,  then it is a tamed \KRf.
  \label{corollary: nucontrol}
 \end{corollary}
 \begin{proof}
   According to Theorem 4.4 of~\cite{CW3}, every weak limit
   $\hat{M}$ is a K\"ahler Ricci soliton orbifold. It has  positive first Chern class
   and it can be embedded into projective space by its plurianticanonical line bundle
   sections (c.f.~\cite{Baily}). In particular, every $\hat{M}$ is a Q-Fano normal variety.
   So Theorem~\ref{theorem: justtamed} applies.
 \end{proof}

 In~\cite{RZZ}, Weidong Ruan, Yuguang Zhang and Zhenlei Zhang
 proved that the Riemannian curvature is uniformly bounded along the
 \KRf if $\int_M |Rm|^n d\mu$ is uniformly bounded.  Under such
 condition, every sequential limit is a smooth K\"ahler Ricci
 soliton manifold, therefore Theorem~\ref{theorem: justtamed} applies and we
 have
  \begin{corollary}
   Suppose $\{(M^n, g(t)), 0 \leq t < \infty\}$ is a \KRf along a
   Fano manifold $M^n$ and $n \geq 3$. If
   \begin{align*}
       \sup_{0 \leq t < \infty} \int_M \snorm{Rm}{g(t)}^n d\mu_{g(t)}
       < \infty,
   \end{align*}
   then $\{(M^n, g(t)), 0 \leq t < \infty\}$ is a tamed \KRf.
  \end{corollary}

\section{\KRF on Fano Surfaces}

 In this section, we give an application of the theorems we
 developed.

\subsection{Convergence of  2-dimensional \KRF}

 As the convergence of 2-dimensional \KRf was studied in~\cite{CW1}
 and~\cite{CW2} for all cases except $c_1^2(M)=1$ or $3$, we will
 concentrate on these two cases in this section.

  \begin{lemma}
     Suppose $M$ is a Fano surface, $S \in H^0(K_M^{-\nu})$, $x \in M$.
  \begin{itemize}
  \item  If $c_1^2(M)=1$, then $\alpha_x(S) \geq \frac{5}{6\nu}$ for
    every $S \in H^0(K_M^{-\nu}),  x \in M$.
  \item  If $c_1^2(M)=3$, then $\alpha_x(S) \geq \frac{2}{3\nu}$ for
  every $S \in H^0(K_M^{-\nu}), x \in M$.  Moreover,
  if $\alpha_x(S_1)=\alpha_x(S_2)=\frac{2}{3\nu}$, then $S_1=\lambda
  S_2$ for some constant $\lambda$.
  \end{itemize}
  \label{lemma: lalpha}
  \end{lemma}

  As a direct corollary, we have

  \begin{lemma}
  Suppose $M$ is a Fano surface, $\nu$ is any positive integer.
  \begin{itemize}
  \item  If $c_1^2(M)=1$, then  $\alpha_{\nu, 1} \geq \frac{5}{6}$.
  \item  If $c_1^2(M)=3$, then  $\alpha_{\nu, 1}= \frac{2}{3}$,
      $\alpha_{\nu, 2}> \frac{2}{3}$.
  \end{itemize}
  \label{lemma: nualpha}
  \end{lemma}

  Because of Lemma~\ref{lemma: nualpha} and Corollary~\ref{corollary: nucontrol},
  we are able to apply Theorem~\ref{theorem: nuconv} and
  Theorem~\ref{theorem: nuconvr} respectively to obtain the
  following theorem.

  \begin{theorem}
    If $M$ is a Fano surface with $c_1^2(M)=1$ or $c_1^2(M)=3$,  then
    the \KRf on M converges to a KE metric exponentially fast.
  \end{theorem}

  Combining this with the result in~\cite{CW1}
  and~\cite{CW2}, we have proved the following result by Ricci flow method.

  \begin{theorem}
    Every Fano surface $M$ has a KRS metric in its canonical class.
    This KRS metric is a KE metric if and only if $Aut(M)$ is
    reductive.
  \end{theorem}

  In particular, we have proved the Calabi conjecture on Fano surfaces by flow
  method.  This conjecture was first proved by Tian in~\cite{Tian90}
  via continuity method.

  \begin{remark}
   In~\cite{Chl}, Cheltsov proved the following fact.
   Unless $M$ is a cubic surface with bad symmetry and with
   Eckardt point (a point passed through by three
   exceptional  lines),  then there exists a finite group $G$ such that
   $\alpha_G(M, \omega)> \frac23$ for every $M$ satisfying
   $c_1^2(M) \leq 5$.   Using this fact, we obtain the convergence
   of \KRf on $M$ directly if $M \sim \Blow{8}$. We thank Tian and Cheltsov for pointing
   this out to us.  However, for the consistency of our own programme,
   we still give an independent proof for the convergence of \KRf on $\Blow{8}$
   without applying this fact.
 \end{remark}

 \subsection{Calculation of Local $\alpha$-invariants}
  In this subsection, we give a basic proof of
  Lemma~\ref{lemma: lalpha}.

  \subsubsection{Local $\alpha$-invariants of  Anticanonical Holomorphic Sections }

  \begin{proposition}
     Let $S \in H^0(\CP^2, 3H)$, $Z(S)$ be the divisor generated by $S$, $x \in Z(S)$.
      Then $\alpha_x(S)$  is totally determined by the singularity
      type of $x$.
      It is classified as in the table~\ref{tab: localalpha}.

  \begin{table}[h]
  \begin{center}
  \begin{tabular}{|c|l|c|}
  \hline
  $\alpha_x(S)$ & Singularity type of $x$ &  $S$'s typical local equation\\
  \hline
  $1$ & smooth & $z$ \\
      \cline{2-3}
      & transversal intersection of two lines & $zw$\\
      \cline{2-3}
      & transversal intersection of a line and a conic curve & $zw$\\
      \cline{2-3}
      & ordinary double point & $z^2-w^2(w+1)$\\
  \hline
  $\frac56$ &  cusp & $z^3 - w^2$ \\
  \hline
  $\frac34$ & tangential intersection of a line and  a conic curve & $z(z+w^2)$\\
  \hline
  $\frac23$ & intersection of three different lines & $zw(z+w)$ \\
  \hline
  $\frac12$ & a point on a double line & $z^2$\\
                 \hline
  $\frac13$ & a point on a triple line & $z^3$ \\
  \hline
  \end{tabular}
  \caption{Local $\alpha$ invariants of holomorphic anticanonical sections on projective plane}
  \label{tab: localalpha}
  \end{center}
  \end{table}
  \end{proposition}
  \begin{proof}
    Direct computation.
  \end{proof}

  \begin{proposition}
     Suppose $M$ to be a Fano surface and $M= \Blow{6}$.
       $S \in H^0(M,K_M^{-1})$. Then  $\alpha_x(S) \geq \frac23$ for every $x \in  Z(S)$.
     Moreover, if both $S_1$ and $S_2 \in H^0(M, K_M^{-1})$,
     $\alpha_x(S_1) = \alpha_x(S_2)= \frac23$. Then there exists a
     nonzero constant $\lambda$ such that $S_1= \lambda S_2$.
  \label{proposition: 6}
  \end{proposition}

   \begin{proof}
    Let $M$ to be $\CP^2$ blowup at points $p_1, \cdots, p_6$ in
    generic positions. Let $\pi: M \to \CP^2$ to be the inverse of blowup process.
     If $S \in H^0(M, K_{M}^{-1})$, then  $\pi_*(Z(S))$ must be a cubic curve $\gamma$
     (maybe reducible) in $\CP^2$ and it must pass through every
     point $p_i$.   It cannot contain any triple line. Otherwise,
     assume it contains a triple line connecting $p_1$ and $p_2$.
     Then $Z(S)= 3H - aE_1-bE_2$ for some $a, b \in \Z^+$. On the
     other hand, we know $Z(S)= 3H- \sum_{i=1}^6 E_i$.
     Contradiction!

      Since no three $p_i$'s are in a same line, similar argument shows that there
      is no double line in  $\pi_*(Z(S))$.

      So the table~\ref{tab: localalpha} implies $\alpha_x(S) = \alpha_{\pi(x)}(\pi_*(S)) \geq  \frac23$
      whenever $\pi(x) \in \CP^2 \backslash \{p_1, \cdots, p_6\}$.
      Therefore we only need to consider singular point $x \in \pi^{-1}(\{p_1, \cdots,
      p_k\})$. Without loss of generality, we assume $x \in \pi^{-1}(p_1)$ and $x$ is a singular point of
      $Z(S)$.      We consider this situation by the singularity type
      of $\pi_*(x)$. Actually, $x$ is a singular point of $Z(S)$  only if $\pi_*(x)$ is a
       singular point of $\pi_*(Z(S))$.  By table~\ref{tab: localalpha},
        we have the  following classification.
       \begin{enumerate}
        \item    $\pi_*(x)=p_1$ is an intersection point of three
        different lines.   This case cannot happen. If such three
        lines exist, one of them must pass through $3$ blowup
        points. Impossible.
        \item   $\pi_*(x)$ is an intersection point of two different
        lines.  In this case, $x$ must be a transversal intersection of
        a curve and the exceptional divisor $E_1$. Therefore,
        $\alpha_x(S)=1$.
        \item   $\pi_*(x)$ is a cusp point. In this case, $x$ must
        be a tangential intersection of a smooth curve and the
        exceptional divisor $E_1$. Moreover, the tangential order is
        just $1$. So $\alpha_x(S)= \frac34$.
        \item   $\pi_*(x)$ is a tangential intersection point of a
        line and a conic curve.   In this case, $x$ is the transversal intersection point
      of three curves   $\gamma_1, \gamma_2,  \gamma_3$. Moreover, they have particular properties.
      $[\gamma_1]=E_1, [\gamma_2] \sim 2H-\sum_{l=1}^6 E_l + E_j,[\gamma_3] \sim H-E_1-E_j$ for
      some $j \in \{2, \cdots, 6\}$.  $x$
      is the intersection of $3$ exceptional lines.
       Clearly,  $\alpha_x(S)=\frac23$.
       \end{enumerate}

    Therefore, no matter whether $x \in \pi^{-1} \{ p_1, \cdots, p_6 \}$,
    we see $\alpha_x(S) \geq \frac23$.  Moreover, $\alpha_x(S)= \frac23$ only if
     $x$ is the transversal intersection of three exceptional lines.

    It is well known that $M$ is a cubic surface,  there are totally $27$ exceptional lines on
    $M$. Every point can be passed through by at most three
    exceptional lines.  Therefore, if $\alpha_x(S_1)= \alpha_x(S_2)=
    \frac23$, then $Z(S_1)=Z(S_2)$ as union of three exceptional lines
    passing through $x$. So there is a nonzero constant $\lambda$
    such that $S_1= \lambda S_2$.

   \end{proof}

  \begin{proposition}
     Suppose $M$ to be a Fano surface and $M \sim \Blow{8}$.
       $S \in H^0(M,K_M^{-1})$. Then  $\alpha_x(S) \geq \frac56$ for every $x \in Z(S)$.
  \label{proposition: 8}
  \end{proposition}

 \begin{proof}
    Same notation as in proof of Proposition~\ref{proposition: 6}, we see
    $\pi_*(Z(S))$ is a cubic curve.
    Suppose $\pi_*(Z(S))$ is reducible, then
    $Z(S)=\gamma_1 + \gamma_2$ with $\gamma_1$ a line and $\gamma_2$ a conic curve.
   So $Z(S)$ can pass at most $2+5=7$ points of the blowup points. On
   the other hand, it must pass through all of them. Contradiction!
   Therefore, $\pi_*(Z(S))$ is irreducible.

  If $\pi(x) \in \CP^2 \backslash \{p_1, \cdots, p_8\}$,
  we have $\alpha_x(S) = \alpha_{\pi(x)}(\pi_*(S)) \geq \frac56 > \frac23$ by
  Table~\ref{tab: localalpha}.
  Suppose $\pi(x)  \in \{p_1, \cdots, p_8\}$.  As the $8$ points are in
  generic position, we know no cubic curve pass through $7$ of them
  with one point doubled.  $\pi_*(Z(S))$ is a cubic curve passing
  all these $8$ points, so it must pass through every point
  smoothly.  As $Z(S)$ is irreducible, $x$ must be a smooth point on
  $Z(S)$. So $\alpha_x(S)=1$.

  In short, $\alpha_x(S) \geq \frac56$.
    \end{proof}

\subsubsection{Local $\alpha$-invariant of Pluri-anticanonical
Holomorphic Sections}

  \begin{proposition}
   If $f, g$ are holomorphic functions (or holomorphic sections of a line bundle) defined in a neighborhood of $x$,
  then $\alpha_x(fg) \geq \frac{\alpha_x(f) \alpha_x(g)}{\alpha_x(f) +
  \alpha_x(g)}$, i.e.,
  \begin{align}
     \frac{1}{\alpha_x(fg)} \leq \frac{1}{\alpha_x(f)} +
     \frac{1}{\alpha_x(g)}.
  \label{eqn: holderalpha}
  \end{align}
  \label{proposition: holderalpha}
  \end{proposition}
  \begin{proof}
   Without loss of generality, we can assume $\alpha_x(f), \alpha_x(g)< \infty$.
   For simplicity of notation, let $a= \alpha_x(f), b=\alpha_x(g),c=\frac{ab}{a+b}$. We only need to prove
    $\alpha_x(fg) \geq c$.

    Fix a small number $\epsilon>0$, note that $\frac{c}{a}+\frac{c}{b}=1$,
  H\"older inequality implies
    \begin{align*}
        \int_U (fg)^{-2c(1-\epsilon)} d\mu = (\int_U
        f^{-2a(1-\epsilon)} d\mu)^{\frac{c}{a}} (\int_U
        g^{-2b(1-\epsilon)})^{\frac{c}{b}}< \infty,
    \end{align*}
  where $U$ is some neighborhood of $x$.
    Therefore, $\alpha_x(fg) \geq c(1-\epsilon)$. As $\epsilon$ can
    be arbitrarily small, we have $\alpha_x(fg) \geq c$.
  \end{proof}

  As an application of Proposition A.1.1 of \cite{Tian90}, we list
  the following property without proof.
  \begin{proposition}
   Suppose $f$ is a holomorphic function vanishing at $x$ with order
   $k$.  In a small neighborhood, we can express $f$ as
   \begin{align*}
     f= a_{ij}z_1^i z_2^j + \cdots
   \end{align*}
   Without loss of generality, we can assume that there is a pair $(i, j)$ such that
   $i \geq j$, $i+j=k$ and $a_{ij} \neq 0$.   Then $\alpha_x(f) \geq
   \frac{1}{i}$.
   \label{proposition: alphacal}
  \end{proposition}

  \begin{lemma}
   Suppose $M$ is a cubic surface, $S \in H^0(K_M^{-m})$, $x \in M$.
   If $\alpha_x(S) \leq \frac{2}{3m}$, then $\alpha_x(S) = \frac{2}{3m}$,
   $S= (S')^m$ where $S' \in
   H^0(K_M^{-1})$ and $Z(S')$ is the union of three lines passing
   through $x$.
  \label{lemma: alpha6}
  \end{lemma}

  \begin{proof}
   We will prove this statement by induction. Suppose we have already proved it for all
   $k \leq m-1$, now we show it is true for $k=m$.

  \setcounter{claim}{0}

   \begin{claim}
     If $S$ splits off an anticanonical holomorphic section $S'$, then $Z(S')$ must be
    a union of three lines passing through $x$. Moreover,
    $S=(S')^m$.
    \label{claim: factorline}
   \end{claim}

   Suppose $S= S' S_{m-1}$ where $S' \in H^0(K_M^{-1})$ and $S_{m-1} \in H^0(K_M^{-(m-1)})$.
   Since $\alpha_x(S')\geq \frac23$ and $\alpha_x(S_{m-1}) \geq \frac{2}{3(m-1)}$
  by induction assumption,  inequality (\ref{eqn: holderalpha})
  implies
 \begin{align*}
 \frac{3m}{2} \leq \frac{1}{\alpha_x(S)} \leq \frac{1}{\alpha_x(S')}
 + \frac{1}{\alpha_x(S_{m-1})} \leq \frac32 + \frac{3(m-1)}{2} =
 \frac{3m}{2}.
 \end{align*}
 It forces that
 \begin{align*}
    \alpha_x(S)= \frac{2}{3m},  \quad \alpha_x(S')=\frac23, \quad
    \alpha_x(S_{m-1}) = \frac{2}{3(m-1)}.
 \end{align*}
 Therefore the induction hypothesis tells us that $Z(S')$ is the
 union of three lines passing through $x$, $S_{m-1}= (S'')^{m-1}$
 and $Z(S'')$ is the union of three lines passing through $x$.
 As there are at most three lines passing through $x$ on a cubic surface, we see
 $Z(S')=Z(S'')$.  By changing coefficients if necessary, we have
 $S_{m-1}=(S')^{m-1}$.  It follows that $S=(S')^m$ and we have finished the proof.

 \begin{claim}
 There must be a line passing through $x$.
 \end{claim}
  Otherwise, there is a pencil of anticanonical divisors
  passing through $x$. In this pencil,
   a generic divisor is irreducible and it vanishes at $x$ with order $2$.
   Choose such a divisor and denote it as $Z(S')$. Locally, we can
   represent $S$ by a holomorphic function $f$, as $\alpha_x(f)=\alpha_x(S) \leq
   \frac{2}{3m}$, we see $mult_x(f) \geq \lceil \frac{3m}{2} \rceil$.  If $m$ is
   odd, then $Z(S') \nsubseteq Z(S)$ will imply
   \begin{align*}
       3m= K_M^{-1} \cdot K_M^{-1} \geq 2 mult_x(f) \geq 3m+1.
   \end{align*}
 Impossible! Since $Z(S')$ is irreducible, we know $Z(S') \subset
 Z(S)$. Therefore, $S=S'S_{m-1}$. According to  Claim~\ref{claim: factorline},
 $Z(S')$ is union of three lines and therefore $Z(S')$ is
 reducible. This contradicts to the assumption of $Z(S')$.

     So $m$ must be an even number and $mult_x(f)=\frac{3m}{2}$ exactly.
 Now $f$ can be written as
  \begin{align*}
       \sum_{i, j\geq 0} a_{ij} z_1^i z_2^j, \quad a_{ij}=0 \quad
       \textrm{whenever} \quad i+j< \frac{3m}{2}.
  \end{align*}
  Using the fact $\alpha_x(f) \leq \frac{2}{3m}$,
  Proposition~\ref{proposition: alphacal} implies $a_{i,j}=0$ whenever $i < \frac{3m}{2}$,
  $a_{\frac{3m}{2}, 0} \neq 0$.
  Therefore locally $f$ can be written as $a_{\frac{3m}{2}, 0}z_1^{\frac{3m}{2}} \cdot h$
  for some nonzero holomorphic function $h$.
  This means that $Z(S)$
  contains  a curve with multiplicity $\frac{3m}{2}$.  This is
  impossible for $S \in K_M^{-m}$ as $M$ is the blown up of six
  generic points on $\CP^2$.


 \begin{claim}
  The number of lines passing through $x$ is greater than $1$.
 \end{claim}
  Otherwise, there is exactly one line $L_1$ passing through $x$.  So
  there is an irreducible degree 2 curve $D$ passing through $x$ such that $L_1+D=Z(S')$
  for some $S' \in H^0(K_M^{-1})$.  Locally, we can write $S$ as $l_1h$ where $l_1$
  is the defining function for $L_1$.  As $\alpha_x(l_1)=1$, H\"older inequality implies
  that $\alpha_x(h) \leq \frac{2}{3m-2}$. Consequently, $mult_x(h) \geq
  \lceil \frac{3m}{2} \rceil -1$.
    If $2L_1 \nsubseteq Z(S)$, we
  have
  \begin{align*}
         m+1=(K_M^{-m}-L_1) \cdot L_1 \geq \{h=0\} \cdot L_1 \geq
         \lceil \frac{3m}{2} \rceil -1 \quad \Leftrightarrow  m \leq 4.
  \end{align*}
  If $m>4$, this inequality is wrong so we have $2L_1 \subset
  Z(S)$. Actually, using this argument and induction, we can show that
  $\lceil  \frac{m}{4} \rceil L_1 \subset Z(S)$.

   For simplicity of notation, let $p= \lceil \frac{m}{4} \rceil$.
  Locally, $S$ can be written as $l_1^p h$. Clearly, $\alpha_x(h) \leq \frac{2}{3m-2p}$
  and $mult_x(h) \geq \lceil \frac{3m}{2} \rceil -p$.
  Let $f_{q}$ and $h_{q-2}$ be the lowest degree term of $f$ and $h$
  respectively. Then we may assume that $h_{q-2}= z_1^{j_1}z_2^{j_2} + \cdots$
  and any term $z_1^i z_2^j$ in $h_{q-2}$ satisfying $i \geq j_1$.
  Now we have two cases to consider.

 \noindent{\textit{Case1. $L_1$ is tangent to $\{z_1 =0\}$.}}

   If $(p+1)L_1 \nsubseteq Z(S)$, then
     \begin{align*}
                     m + p=(K_M^{-m}- p L_1) \cdot L_1 \geq \{h=0\} \cdot L_1
      \geq (\lceil \frac{3m}{2} \rceil -p) \cdot 2, \quad \Leftrightarrow
       p \geq \frac{m+2\lceil \frac{m}{2} \rceil}{3}.
   \end{align*}
   Here we use the fact $j_1 \geq \lceil \frac{3m}{2} \rceil -p$ since $\alpha_x(h) \leq
   \frac{2}{3m-2p}$. This contradicts to our definition $p=\lceil \frac{m}{4} \rceil$.
   Therefore, $(p+1)L_1 \subset Z(S)$.

 \noindent{\textit{Case2. $L_1$ is not tangent to $\{z_1=0\}$.}}

  In this case, $f_q= \lambda z_1^{a_1} z_2^{a_2+p} + \cdots$
  for some $\lambda \neq 0$.  Moreover, every $z_1^i z_2^j$ in $f_q$
  satisfies $i  \geq j_1$.   Therefore, the fact $\alpha_x(S) \leq \frac{2}{3m}$
  and Proposition~\ref{proposition: alphacal} implies $a_1 \geq
  \lceil \frac{3m}{2} \rceil$. It follows that $q \geq p + \lceil \frac{3m}{2} \rceil$.
  Under these conditions, if $(p+1)L_1 \nsubseteq Z(S)$, we have
  \begin{align*}
        m+p=(K_M^{-m} - p L_1) \cdot L_1 \geq  \lceil \frac{3m}{2}
        \rceil
        \Leftrightarrow p \geq \lceil \frac{m}{2} \rceil.
  \end{align*}
  Impossible if $m \geq 3$.  So $(p+1)L_1 \subset Z(S)$.
  If $m \leq 2$, as $\lceil \frac{m}{2} \rceil = \lceil \frac{m}{4}
  \rceil=1$, we already know $\lceil \frac{m}{2} \rceil L_1 \subset
  Z(S)$.

  Therefore by repeatedly rewriting $S$ in local charts and considering case 1 and case
  2,  we can actually prove that and  $\lceil \frac{m}{2} \rceil L_1 \subset Z(S)$.
  For simplicity, let $n=\lceil \frac{m}{2} \rceil$.
  Moreover, we have following conditions:

  \textit{ Suppose $S$ can be written
  as $l_1^n h'$ locally. Then either $(n+1)L_1 \subset Z(S)$ or
  $L_1$ is not tangent to $\{z_1=0\}$.}

  From here, we can show $D \subset Z(S)$.
  In fact, if $(n+1)L_1 \subset Z(S)$ and $D \nsubseteq Z(S)$,
  we have
  \begin{align*}
   2m =K_M^{-m} \cdot D \geq (n+1)L_1 \cdot D+  mult_x(h'') \geq 2(n+1) + (\lceil \frac{3m}{2} \rceil -(n+1))
   \Leftrightarrow  m \geq 2n+1,
  \end{align*}
  where $h''$ is the function such that locally $S$ is represented by
  $l_1^{n+1}h''$.
  This inequality is impossible as $n=\lceil \frac{m}{2} \rceil$.
  If $L_1$ is not tangent to $\{z_1=0\}$, we know
  $mult_x(h') \geq n + \lceil \frac{3m}{2} \rceil= m+2n$.
  Therefore, $D \nsubseteq Z(S)$ implies that
  \begin{align*}
      2m=K_M^{-m}\cdot D \geq nL_1 \cdot D + mult_x(h') \geq m+4n
      \Leftrightarrow  m \geq 4 \lceil \frac{m}{2} \rceil.
  \end{align*}
  Impossible!   Therefore,  no matter which case happens, we have $D \subset Z(S)$. So $D+L_1 \subset Z(S)$.
  It follows that $S$ splits off an $S' \in H^0(K_M^{-1})$ with
  $Z(S')=L_1+D$, this contradicts to Claim~\ref{claim: factorline}!

 \begin{claim}
 The number of lines passing through $x$ is greater than $2$.
 \label{claim: twolines}
 \end{claim}

  Otherwise, there are only two lines $L_1$ and $L_2$ passing through
  $x$.  There is a unique line $L_3$ not passing through $x$ such that
  $L_1+L_2+L_3 \in K_M^{-1}$.  We first prove the following property:
 \begin{align*}
 k(L_1 + L_2) \subset Z(S)
 \quad \textrm{for all}  \quad 0 \leq k \leq n=\lceil \frac{m}{2}\rceil.
 \end{align*}
 Actually, by induction, we can assume $(k-1)(L_1+L_2) \in Z(S)$. Then $S$ can
be represented
   by a holomorphic function $f=l_1^{k-1} l_2^{k-1} h$ locally.
  Note that $\alpha_x(l_1^{k-1}  l_2^{k-1})=\frac{1}{k-1}$, H\"older inequality implies
  $\alpha_x(h) \leq \frac{\frac{2}{3m}}{1-\frac{2(k-1)}{3m}} = \frac{2}{3m-2(k-1)}$.
  It follows that
  \begin{align*}
  mult_x(h) \geq \lceil \frac{3m}{2} \rceil + 1-k= m+n+1-k.
  \end{align*}
  If $kL_1 \nsubseteq Z(S)$, we
  have
  \begin{align*}
          m=(K_M^{-m}- (k-1)(L_1+L_2)) \cdot L_1 \geq \{h=0\} \cdot {l_1=0}
           \geq  mult_x(h) \geq m+n + 1-k  \Leftrightarrow  k \geq
           n+1.
  \end{align*}
 This contradicts to the  assumption of $k$.   Therefore, we have $kL_1 \subset Z(S)$.
 Similarly, $kL_2 \subset Z(S)$.  So $k(L_1 + L_2) \subset Z(S)$.

  Now locally $S$ can be written as $l_1^{n} l_2^n h$. We have
  $\alpha_x(h)  \leq  \frac{2}{3m-2n}$, $mult_x(h) \geq \lceil \frac{3m}{2}-n \rceil=m$.
  Assume $mult_x(h)=m$.  Under a local
  coordinates, $h=\sum_{i, j \geq 0} a_{ij} z_1^i z_2^j$.
  According to the fact $\alpha_x(h)  \leq \frac{2}{3m-2n}$,
  Proposition~\ref{proposition: alphacal} implies that $a_{ij}=0$ whenever $i< \lceil \frac{3m}{2}-n \rceil=m$.
  Since $mult_x(h)=m$, we see that the
  lowest homogeneous term of $f$ is of form $l_1^{n} l_2^{n} z_1^{m}$.
  The condition $\alpha_x(S) \leq \frac{2}{3m} < \frac{1}{m}$
  implies that either $L_1$ or $L_2$ is tangent to $\{z_1=0\}$ at
  $x$.  Suppose $L_1$ does so.   If $(n+1)L_1 \nsubseteq Z(S)$, we have
  \begin{align*}
           m+n&=(K_M^{-m} - nL_1) \cdot L_1 \geq \{l_2^n h=0 \} \cdot
           L_1\\
             &\geq n + \{\sum_{i,j \geq 0} a_{ij}z_1^iz_2^j=0\} \cdot L_1
           \geq n + \inf\{2i+j | a_{ij} \neq 0\} \geq n + 2m.
  \end{align*}
  Impossible! It follows that $ (n+1)L_1 + nL_2 \subset Z(S)$.

  Consider $L_3$. If $L_3 \nsubseteq Z(S)$, we have
  \begin{align*}
       m=L_3 \cdot K_M^{-m} \geq  ((n+1)L_1 + nL_2)= 2n+1.
  \end{align*}
  This absurd inequality implies $L_3 \subset Z(S)$. Let $S' \in K_M^{-1}$ such that $Z(S')=L_1+ L_2+ L_3$.
  So have split $S$ as $S= S' S_{m-1}$. However, $Z(S')$ is not the union
  of three lines passing through $x$.  This contradicts to Claim~\ref{claim: factorline}!\\

 So there must exist three lines $L_1, L_2, L_3 \subset Z(S)$ passing through $x$.
 Since $M$ is a cubic surface, there exists an $S' \in H^0(K_M^{-1})$
 such that $Z(S')=L_1+L_2+L_3$.
 As we argued in Claim~\ref{claim: twolines}, $L_1, L_2, L_3 \subset
 Z(S)$. Therefore $L_1+L_2+L_3 \subset Z(S)$ and $S$ splits off an
 anticanonical holomorphic section $S'$.
  By Claim~\ref{claim: factorline}, we have $S=(S')^m$.
  \end{proof}

 Similarly, we can prove the following property by induction.
  \begin{lemma}
   Suppose $M$ is a Fano surface and $M \sim \Blow{8}$ , $S \in H^0(K_M^{-m})$, $x \in M$.
   Then $\alpha_x(S) \geq \frac{5}{6m}$ for every $x \in M$.
  \label{lemma: alpha8}
  \end{lemma}
 \begin{proof}
   Suppose we have proved this statement for all $k \leq m-1$.

   Suppose this statement doesn't hold for $k=m$, then there is a
   holomorphic section $S \in H^0(K_M^{-m})$ and point $x \in M$
   such that $\alpha_x(S) < \frac{5}{6m}$.  Let $f$ be a local
   holomorphic function representing $S$. Clearly, $mult_x(f) >
   \frac{6m}{5}$.
   Choose $S' \in H^0(K_M^{-1})$
   such that $x \in Z(S')$.  Since $S'$ is irreducible, if $Z(S') \nsubseteq
   Z(S)$, we have
   \begin{align*}
               m= Z(S) \cdot Z(S')> \frac{6m}{5}.
   \end{align*}
   It is impossible!   Therefore, $Z(S') \subset Z(S)$. It follows
   that $S=S'S_{m-1}$ for some $S_{m-1} \in H^0(K_M^{-1})$.  So
   Proposition~\ref{proposition: alphacal} implies
   \begin{align*}
      \alpha_x(S) \geq \frac{\frac56 \cdot \frac{5}{6(m-1)}}{\frac56 + \frac{5}{6(m-1)}}
        =\frac{5}{6m}.
   \end{align*}
   This contradicts to the assumption of $\alpha_x(S)$!
 \end{proof}

 Lemma~\ref{lemma: lalpha} is the combination of Lemma~\ref{lemma: alpha6} and Lemma~\ref{lemma: alpha8}.

\vspace{1in}

 Xiuxiong Chen,  Department of Mathematics, University of
 Wisconsin-Madison, Madison, WI 53706, USA; xiu@math.wisc.edu\\

 Bing  Wang, Department of Mathematics, University of Wisconsin-Madison,
 Madison, WI, 53706, USA; bwang@math.wisc.edu

 \qquad \qquad \quad \; Department of Mathematics, Princeton University,
  Princeton, NJ 08544, USA; bingw@math.princeton.edu

\end{document}